%% file: stabregopt.tex
\documentclass[10pt]{article}
\usepackage[hmargin=1.0in]{geometry}
\usepackage{amsmath,amsthm,amsfonts,graphicx,subfigure,color,multicol,xspace}
\usepackage[disable]{todonotes}
\usepackage{fancybox}
\usepackage{url}
\usepackage{algpseudocode}
\usepackage{algorithm}

\DeclareMathOperator*{\maximize}{maximize}
\DeclareMathOperator*{\minimize}{minimize}

\begin{document}
\input{macros}
\title{Optimal stability polynomials for numerical integration of initial value problems}
\author{David I. Ketcheson\thanks{King Abdullah University of Science and Technology (KAUST),
    Division of Mathematical and Computer Sciences and Engineering,
    Thuwal, 23955-6900. Saudi Arabia  (\mbox{david.ketcheson@kaust.edu.sa}) } \and 
    Aron J. Ahmadia\thanks{King Abdullah University of Science and Technology (KAUST),
    Division of Mathematical and Computer Sciences and Engineering,
    Thuwal, 23955-6900. Saudi Arabia}}

\maketitle

\begin{abstract}

We consider the problem of finding optimally stable polynomial approximations
to the exponential for application to one-step integration of initial value
ordinary and partial differential equations.  The objective is to find
the largest stable step size and corresponding method for a given problem
when the spectrum of the initial value problem is known.
The problem is expressed in 
terms of a general least deviation feasibility problem.  Its solution is
obtained by a new fast, accurate, and robust algorithm based on convex
optimization techniques.  Global convergence of the algorithm is proven 
in the case that the order of approximation is one and in the case that
the spectrum encloses a starlike region.  Examples demonstrate the
effectiveness of the proposed algorithm even when these conditions are 
not satisfied.
\end{abstract}

\section{Stability of Runge--Kutta methods}
Runge--Kutta methods are among the most widely used types of numerical integrators
for solving initial value ordinary and partial differential equations.
The time step size should be taken as large as possible since the cost of
solving an initial value problem (IVP) up to a
fixed final time is proportional to the number of steps that must be taken.
In practical computation, the time step is often limited by 
stability and accuracy constraints.  Either accuracy, stability, or both may
be limiting factors for a given problem; see e.g. \cite[Section 7.5]{rjl:fdmbook}
for a discussion.  The linear stability and accuracy of an explicit Runge--Kutta
method are characterized completely by the so-called stability polynomial of the method, which in turn dictates
the acceptable step size \cite{Butcher_2008,Hairer_Wanner_1993}.
In this work we present an approach for constructing a stability polynomial
that allows the largest absolutely stable step size for a given problem.  

In the remainder of this section, we review the stability concepts for Runge--Kutta
methods and formulate the stability optimization problem.
Our optimization approach, described in Section \ref{sec:algorithm}, is based on
reformulating the stability optimization problem 
in terms of a sequence of convex subproblems and using bisection.
We examine the theoretical properties of the proposed algorithm and
prove its global convergence for two important cases.

A key element of our optimization algorithm is the use of numerical convex optimization
techniques.  We avoid a poorly conditioned numerical formulation by posing the problem
in terms of a polynomial basis that is
well-conditioned when sampled over a particular region of the complex plane.
These numerical considerations, which become particularly important when the
number of stages of the method is allowed to be very large, 
are discussed in Section \ref{sec:implementation}.

In Section \ref{sec:examples} we apply our algorithm to several examples of complex spectra.
Cases where optimal results are known provide verification of the algorithm,
and many new or improved results are provided.

Determination of the stability polynomial
is only half of the puzzle of designing optimal explicit Runge--Kutta methods.
The other half is the determination of the Butcher coefficients.
While simply finding methods
with a desired stability polynomial is straightforward, many additional challenges
arise in that context; for instance, additional nonlinear order conditions,
internal stability, storage, and embedded error estimators.
The development of full Runge--Kutta methods based on optimal stability
polynomials is the subject of ongoing work \cite{parsani2012}.


\subsection{The stability polynomial}
A linear, constant-coefficient initial value problem takes the form
\begin{align} \label{eq:ivp}
u'(t) & = Lu  & u(0) & = u_0,
\end{align}
where $u(t) : \Real\to\Real^N$ and $L \in\Real^{N\times N}$.  
When applied to the linear IVP \eqref{eq:ivp},
any Runge--Kutta method reduces to an iteration of the form
\begin{align} \label{eq:RKiter}
u_n = R(hL)u_{n-1},
\end{align}
where $h$ is the step size and $u_n$ is a numerical approximation to $u(nh)$.
The {\em stability function} $R(z)$ depends only on the coefficients of the
Runge--Kutta method \cite[Section 4.3]{Gottlieb2011a}\cite{Butcher_2008,Hairer_Wanner_1993}.  
In general, the stability function of an $s$-stage explicit Runge-Kutta method is
a polynomial of degree $s$ 
\begin{align} \label{eq:polyform}
R(z) & = \sum_{j=0}^s a_j z^j.
\end{align}
Recall that the exact solution of \eqref{eq:ivp} is $u(t) = \exp(tL)u_0$.  Thus,
if the method is accurate to order $p$, the stability polynomial must be
identical to the exponential function up to terms of at least order $p$:
\begin{align} \label{eq:oc}
a_j = \frac{1}{j!} \ \mbox{ for } \ \ \ \ \ 0\le j \le p.
\end{align}

\subsection{Absolute stability}
The stability polynomial governs the local
propagation of errors, since any perturbation to the solution will be multiplied
by $R(z)$ at each subsequent step.  The propagation of errors thus
depends on $\|R(hL)\|$, which leads us to define the {\em absolute stability region}
\begin{align}
S = \{ z \in \Complex : |R(z)|\le 1 \}.
\end{align}
For example, the stability region of the classical fourth-order method 
is shown in Figure \ref{fig:rk4}.

Given an initial value problem \eqref{eq:ivp}, let $\Lambda\in\Complex$
 denote the spectrum of the matrix $L$.
We say the iteration \eqref{eq:RKiter} is absolutely stable if 
\begin{align} \label{eq:spec-cond}
h\lambda \in S & \ \ \mbox{ for all } \ \ \lambda \in \Lambda.
\end{align}
Condition \eqref{eq:spec-cond} implies that
$u_n$ remains bounded for all $n$.  More importantly, 
\eqref{eq:spec-cond} is a necessary condition for stable
propagation of errors\footnote{For non-normal $L$, it may be important to consider 
the pseudospectrum rather than the spectrum; see Section \ref{sec:pseudospectrum}.}.
Thus the maximum stable step size is given by
\begin{align}
h_\textup{stable} = \max \{ h\ge0 : |R(h\lambda)|\le 1 \mbox{ for } \lambda\in\Lambda\}.
\end{align}

As an example, consider the advection equation 
$$\frac{\partial}{\partial t} u(x,t) + \frac{\partial}{\partial x} u(x,t) = 0, \ \ \ \ \ x\in(0,M),$$
discretized in space by first-order upwind differencing with spatial mesh size $\Dx$
$$U_i'(t) = - \frac{U_i(t)-U_{i-1}(t)}{\Dx} \ \ \ \ \ \ \ \ 0\le i \le N$$
with periodic boundary condition $U_0(t)=U_N(t)$.
This is a linear IVP \eqref{eq:ivp} with $L$ a circulant bidiagonal matrix.
The eigenvalues of $L$ 
are plotted in Figure \ref{fig:updiff} for $\Dx=1, N=M=20$.  
To integrate this system with the classical fourth-order Runge--Kutta method, 
the time step size must
be taken small enough that the scaled spectrum $\{h\lambda_i\}$
lies inside the stability region. Figure \ref{fig:rk4updiff} shows
the (maximally) scaled spectrum superimposed on the stability region.

The motivation for this work is that a larger stable step size can
be obtained by using a Runge--Kutta method with a larger region of absolute stability.
Figure \ref{fig:opt104updiff} shows the stability region of an optimized
ten-stage Runge--Kutta method of order four that allows a
much larger step size.  The ten-stage method was obtained using the technique that
is the focus of this work.
Since the cost of taking one step is typically proportional to the number of
stages $s$, we can compare the efficiency of methods with different numbers of
stages by considering the {\em effective step size} $h/s$. 
Normalizing in this manner, it turns out that the ten-stage method is nearly twice as fast
as the traditional four-stage method.

\begin{figure}
  \center
  \subfigure[\label{fig:updiff}]{\includegraphics[height=2.5in]{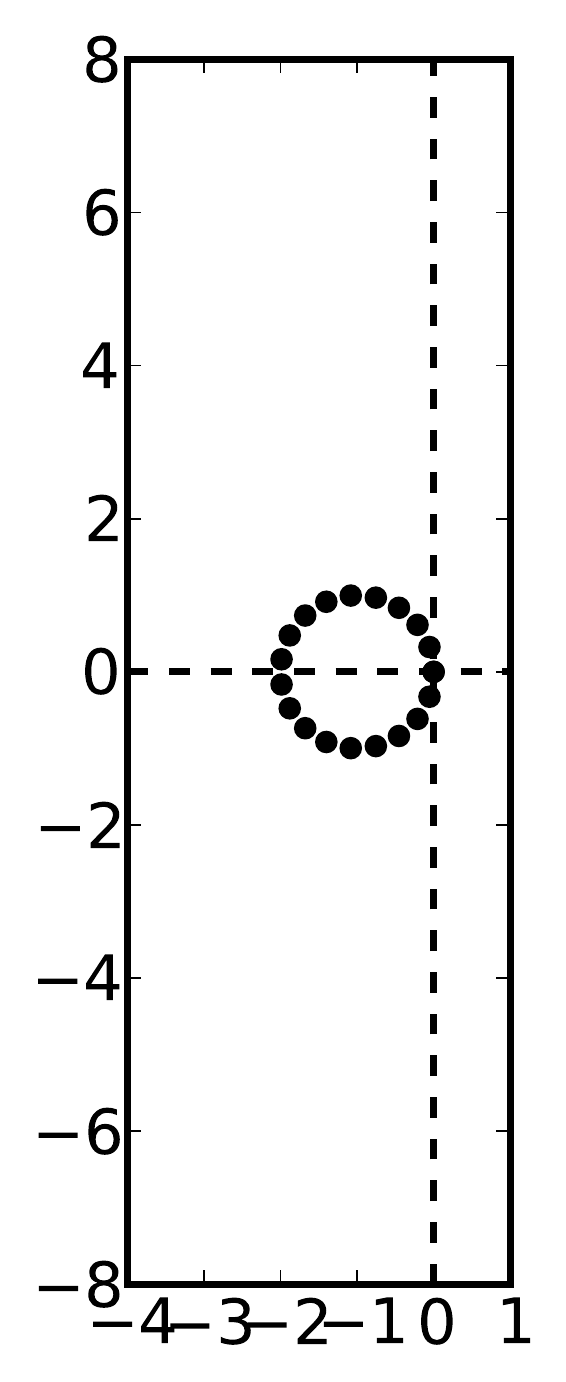}} 
  \subfigure[\label{fig:rk4}]{\includegraphics[height=2.5in]{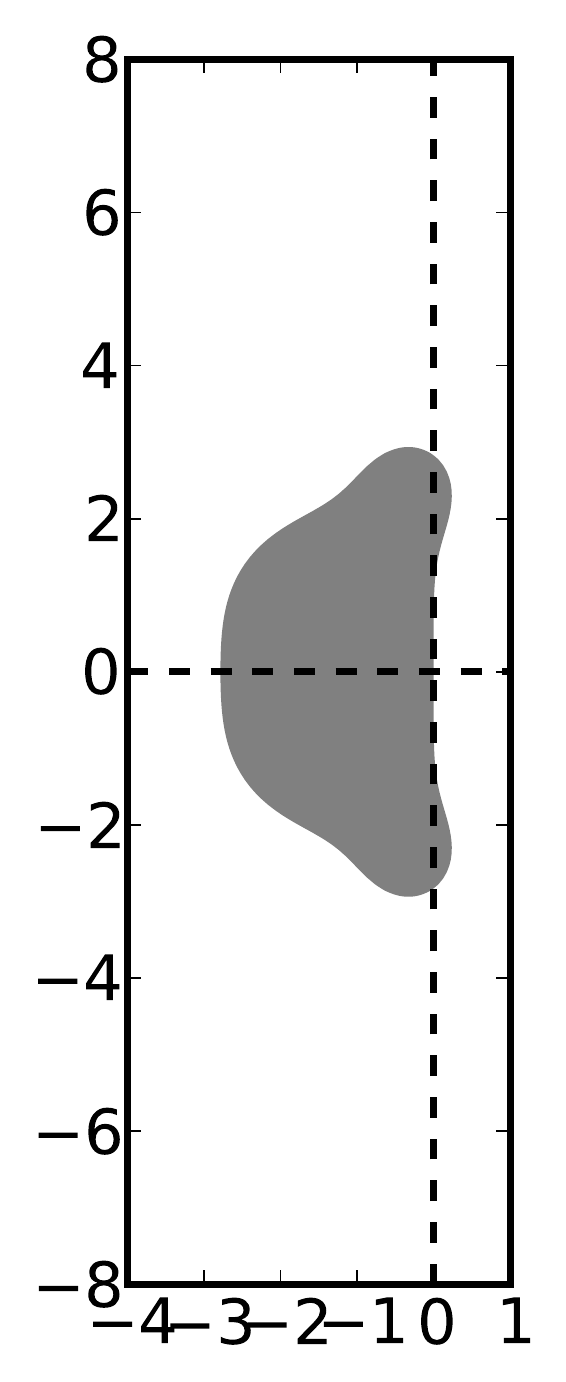}} 
  \subfigure[\label{fig:rk4updiff}]{\includegraphics[height=2.5in]{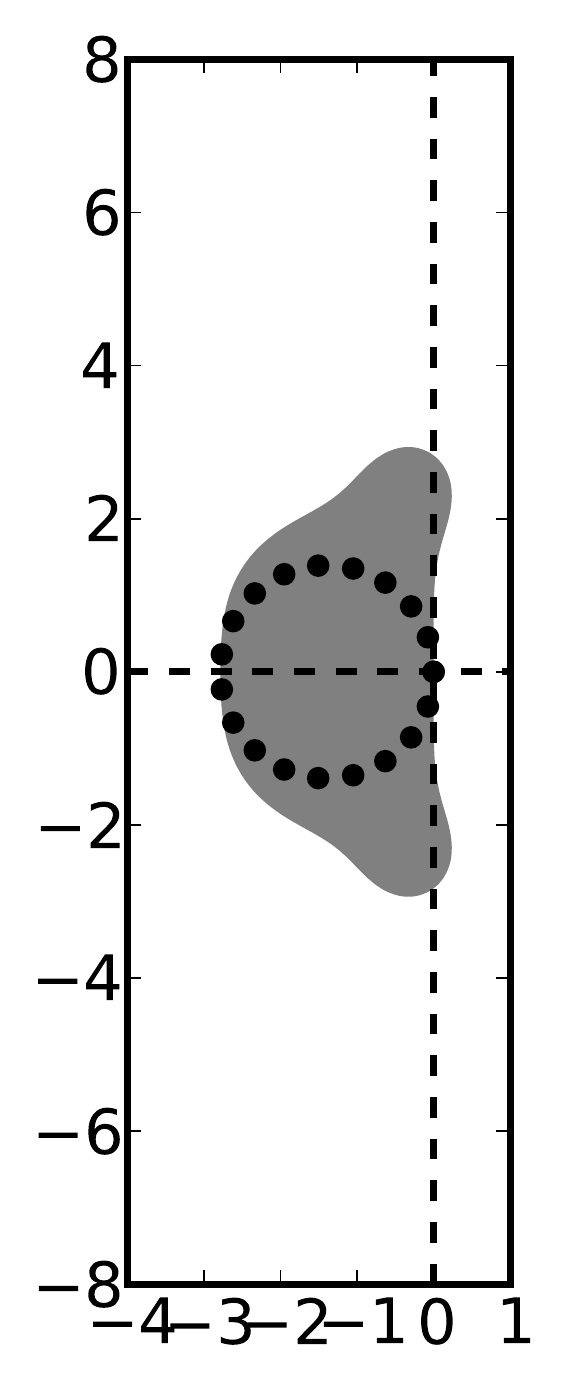}} 
  \subfigure[\label{fig:opt104updiff}]{\includegraphics[height=2.5in]{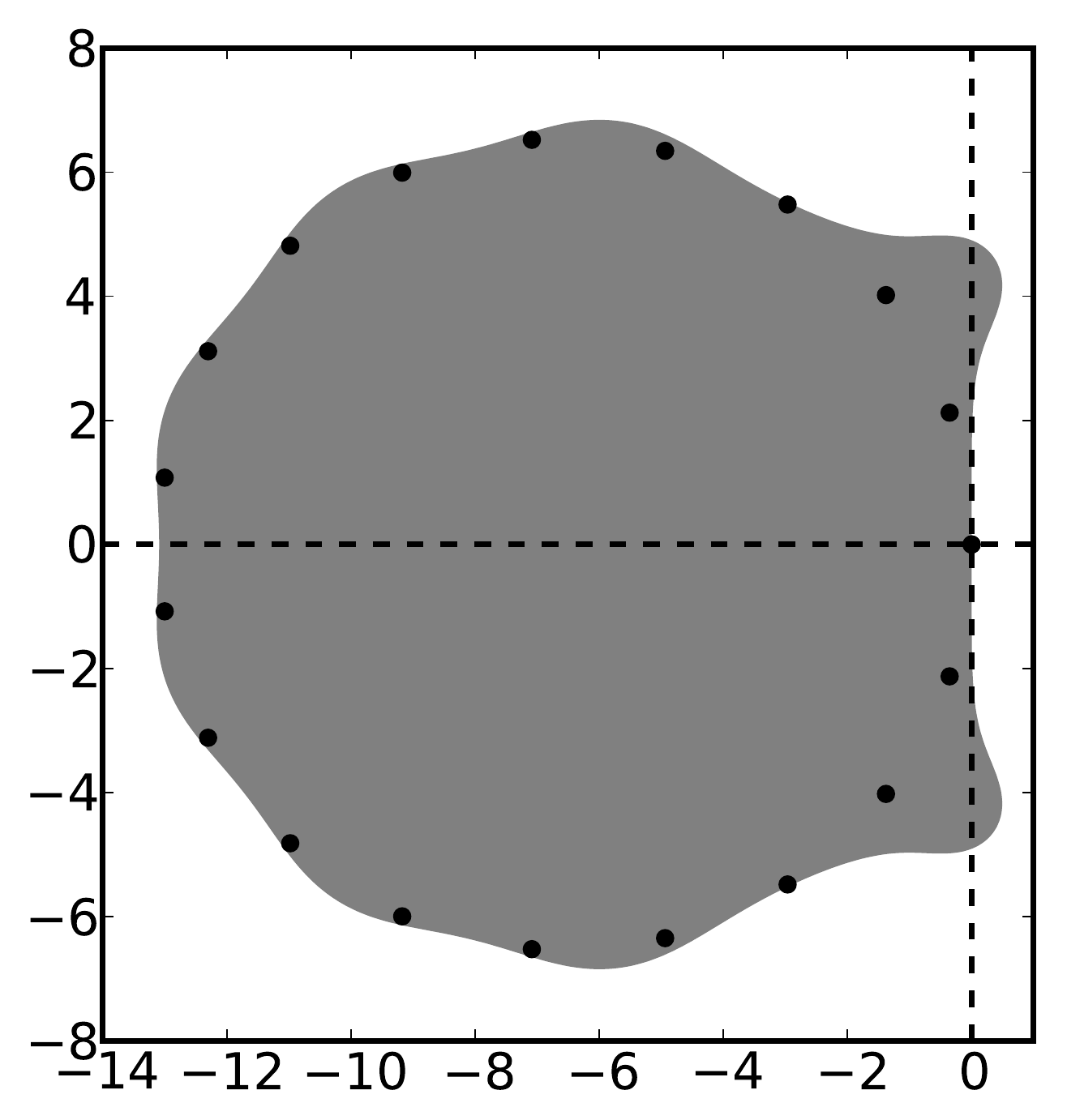}}
  \caption{(a) spectrum of first-order upwind difference matrix using $N=20$ points in space; (b) stability region of the classical fourth-order Runge--Kutta method; (c) Scaled spectrum $h\lambda$ with $h=1.39$; (d) Scaled spectrum $h\lambda$  for optimal 10-stage method with $h=6.54$.}
\end{figure}

\subsection{Design of optimal stability polynomials\label{sec:formulation}}
We now consider the problem of choosing a stability polynomial so as to maximize
the step size under which given stability constraints are satisfied.
The objective function $f(x)$ is simply the step size $h$.  The stability conditions
yield nonlinear inequality constraints.
Typically one also wishes to impose a minimal order of accuracy.
The monomial basis representation \eqref{eq:polyform} of $R(z)$ is then
convenient because the first $p+1$
coefficients $\{a_0,a_1,\dots,a_p\}$ of the stability polynomial are simply taken to
satisfy the order conditions \eqref{eq:oc}.
As a result, the space of decision variables has dimension $s+1-p$,
and is comprised of the coefficients $\{a_{p+1},a_{p+2},\dots,a_s\}$,
as well as the step size $h$.  
Then the problem can be written as
\begin{problem}[stability optimization] \label{prob:stab_opt}
  Given $\Lambda \subset \Complex$, order $p$, and number of stages $s$,
  \begin{equation*}
    \begin{aligned}
      & \maximize_{a_{p+1},a_{p+2},\dots,a_s,h} & & h \\
      & \textnormal{subject to} & & |R(h\lambda)| -1 \le 0, \ \ \ \ \ \forall \lambda\in\Lambda.
    \end{aligned}
  \end{equation*}
\end{problem}
We use $\Hopt$ to denote the solution of Problem \ref{prob:stab_opt} (the optimal
step size) and $\Ropt$ to denote the optimal polynomial.


The set $\Lambda$ may be finite, corresponding to a finite-dimensional ODE system
or PDE semi-discretization, or infinite (but bounded), corresponding to a PDE or perhaps its
semi-discretization in the limit of infinitesimal mesh width.
In the latter case, Problem \ref{prob:stab_opt} is a semi-infinite program (SIP).
In Section \ref{sec:examples} we approach this by using a finite discretization of
$\Lambda$; for a discussion of this and other approaches to
semi-infinite programming, see \cite{Hettich1993}.

\subsection{Previous work}
The problem of finding optimal stability polynomials 
is of fundamental importance in the numerical solution of
initial value problems, and its solution or approximation has been studied by
many authors for several decades
\cite{Lawson1966,Riha1972a,vanderHouwen1972,jeltsch1978largest,jeltsch1981,jeltsch1982stability,vichnevetsky1983,kinnmark1984a,kinnmark1984b,pike1985,ingemar1986,Renaut1990,vanderhouwen1996,Medovikov1998,Mead1999,abdulle2000roots,abdulle2001second,abdulle2002fourth,verwer2004rkc,Bogatyrev2005a,torrilhon2007essentially,Sommeijer2007,Bernardini2009,Allampalli2009,martin2009second,Niegemann2011,toulorge2011b,Skvortsov2011a}.
Indeed, it is closely related to the problem of finding polynomials of least
deviation, which goes back to the work of Chebyshev.  
A nice review of much of the early work on Runge--Kutta stability regions
can be found in \cite{vanderhouwen1996}.
The most-studied cases are those where the eigenvalues lie on the negative real
axis, on the imaginary axis, or in a disk of the form $|z+w|\le w$.
Many results and optimal polynomials, both
exact and numerical, are available for these cases.  Much less is available
regarding the solution of Problem \ref{prob:stab_opt} for arbitrary spectra $\lambda_i$.


Two very recent works serve to illustrate both the progress that has been
made in solving these problems with nonlinear programming, and the challenges
that remain.  In \cite{toulorge2011b}, optimal schemes are sought for integration
of discontinuous Galerkin discretizations of wave equations, where the optimality
criteria considered include both accuracy and stability measures.  
The approach used there is based on sequential quadratic programming (local optimization)
with many initial guesses.  The authors
consider methods of at most fourth order and situations with $s-p\le 4$ 
``because the cost of the optimization procedure becomes prohibitive for a
higher number of free parameters.''  In \cite{Niegemann2011}, optimally stable
polynomials are found for certain spectra of interest for $2\le p\le 4$ and (in
a remarkable feat!) $s$ as large as $14$.  The new methods obtained achieve
a 40-50\% improvement in efficiency for discontinuous Galerkin integration of
the 3D Maxwell equations.
The optimization approach employed therein is again a direct search algorithm 
that does not guarantee a globally optimal solution but 
``typically converges ... within a few minutes''.  However, it was apparently unable to 
find solutions for $s>14$ or $p>4$.  The method we present in the next section can rapidly find solutions
for significantly larger values of $s,p$, and is provably globally convergent under 
certain assumptions (introduced in section \ref{sec:algorithm}).

\section{An efficient algorithm for design of globally optimal stability
            polynomials\label{sec:algorithm}}
Evidently, finding the global solution of Problem \ref{prob:stab_opt} is
in general quite challenging.  
Although the Karush-Kuhn-Tucker (KKT) conditions provide
necessary conditions for optimality in the solution of nonlinear programming problems, the stability constraints in Problem \ref{prob:stab_opt}
are nonconvex, hence suboptimal local minima may exist.  

\subsection{Reformulation in terms of the least deviation problem\label{sec:reformulation}}
The primary theoretical advance leading to the new results in this paper
is a reformulation of Problem \ref{prob:stab_opt}.
Note that Problem \ref{prob:stab_opt} is (for $s>2$) nonconvex
since $R(h\lambda)$ is a nonconvex function in $h$.

Instead of asking for the maximum stable step size we now ask,
for a given step size $h$, how small the maximum modulus of $R(h\lambda)$ can be.
This leads to a generalization of the classical least deviation problem.
\begin{problem}[(Least Deviation)] \label{prob:unconstrained}
  Given $\Lambda\subset \Complex$, $h\in\Real^+$ and $p,s\in\mathbb{N}$
  \begin{equation*}
    \begin{aligned}
      & \minimize_{a_{p+1},a_{p+2},\dots,a_s} && \max_{\lambda\in\Lambda}\left(|R(h\lambda|-1\right). 
    \end{aligned}
  \end{equation*} 
\end{problem}
We denote the solution of Problem \ref{prob:unconstrained} by $r_{p,s}(h,\Lambda)$,
or simply $r(h,\Lambda)$.
Note that $|R(z)|$ is convex with respect to $a_j$, 
since $R(z)$ is linear in the $a_j$.  Therefore, Problem \ref{prob:unconstrained} is convex. 
Furthermore, Problem \ref{prob:stab_opt} can be formulated in terms of Problem \ref{prob:unconstrained}.
\begin{problem}[Reformulation of Problem \ref{prob:stab_opt}] \label{prob:sequence}
  Given $\Lambda\subset \Complex$, and $p,s\in\mathbb{N}$,
    \begin{equation*}
      \begin{aligned}
        & \maximize_{a_{p+1},a_{p+2},\dots,a_s} && h \\
        & \textnormal{subject to} && r_{p,s}(h,\Lambda) \le 0.
      \end{aligned}
    \end{equation*}
  \end{problem}

\subsection{Solution via bisection}
Although Problem \ref{prob:sequence} is not known to be convex, it is an
optimization in a single variable.  It is natural then to apply a bisection
approach, as outlined in Algorithm \ref{alg:bisection}.

\begin{algorithm}\caption{Simple bisection}
\label{alg:bisection}
\begin{algorithmic}

\State $\hmin = 0$
\While {$\hmax-\hmin>\epsilon$}
    \State $h = (\hmax+\hmin)/2$
    \State Solve Problem \ref{prob:unconstrained}
    \If {$r_{p,s}(h,\Lambda)\le0$}
        \State $\hmin = h$
    \Else
        \State $\hmax = h$
    \EndIf
\EndWhile \\
\Return $H_\epsilon = \hmin$
\end{algorithmic}
\end{algorithm}
As long as $\hmax$ is chosen large enough, it is clear that $r(h,\Lambda)=0$
for some $h\in[\hmin,\hmax]$.  Global convergence of the algorithm is assured
only if the following condition holds:
\begin{align} \label{h-monotone} 
r_{p,s}(h_0,\Lambda)=0 \implies r_{p,s}(h,\Lambda)\le 0 \text{ for all } 0\le h\le h_0.  
\end{align}
We now consider conditions under which condition \eqref{h-monotone} can be
established.  We have the following important case.

\begin{theorem}[Global convergence when $p=1$] \label{theorem-bisection}
    Let $p=1$, $\Lambda\subset\Complex$ and $s\ge 1$.  Take $\hmax$ large enough
    so that $r(\hmax,\Lambda)>0$.
    Let $\Hopt$ denote the solution
    of Problem \ref{prob:stab_opt}.  Then the output of Algorithm \ref{alg:bisection}
    satisfies $$\lim_{\epsilon\to0} H_\epsilon = \Hopt.$$
\end{theorem}

\newcommand{\cvxnew}[1] {#1}

\begin{proof}
  Since $r(0,\Lambda)=0 < r(\hmax,\Lambda)$ and $r(h,\Lambda)$ is continuous in $h$,
  it is sufficient to prove that condition \eqref{h-monotone} holds.
  We have $|\Ropt(\Hopt\lambda)|\le 1$ for all $\lambda\in\Lambda$.
  We will show that there exists $R_\mu(z) = \sum_{j=0}^s a_j(\mu) z^j$ such
  that $a_0=a_1=1$ and $$|R_\mu(\mu \Hopt\lambda)|\le 1 \ \ \ \forall \lambda\in\Lambda, \ \  0\le\mu\le1.$$
  
  Let $\hat{a}_j$ be the coefficients of the optimal polynomial:
  \begin{equation*}
    \Ropt(z) = 1 + z + \sum_{j=2}^s\hat{a}_j z^j,
  \end{equation*}
  and set
  \begin{equation*}
    a_j(\mu) = \mu^{1-j}\hat{a_j}.
  \end{equation*}
  Then
  \begin{align*}
    R_\mu(\mu\Hopt\lambda) & = 1 + \mu\Hopt\lambda + \sum_{j=2}^s\mu^{1-j}\hat{a}_j \mu^j\Hopt^j \lambda^j =
    1 + \mu\left(\sum_{j=1}^s \hat{a}_j \Hopt^j \lambda^j\right) \\
    & = 1 + \mu(\Ropt(\Hopt\lambda)-1),
  \end{align*}  
  where we have defined $\hat{a}_1 = 1$.
  Define $g_\lambda(\mu)=R_\mu(\mu\Hopt\lambda)$.  Then $g_\lambda(\mu)$ is linear in $\mu$ and has the
  property that, for $\lambda\in\Lambda$, $|g_\lambda(0)| = 1$ 
  and $|g_\lambda(1)| \le 1$ (by the definition of $\Hopt,\Ropt$).  Thus by convexity $|g(\mu)|\le 1$
  for $0\le\mu\le1$.
\end{proof}

For $p>1$, condition \eqref{h-monotone} does not necessarily hold.
For example, take $s=p=4$; then the stability polynomial \eqref{eq:polyform}
is uniquely defined as the degree-four Taylor approximation of the exponential,
corresponding to the classical fourth-order Runge--Kutta method 
that we saw in the introduction.
Its stability region is plotted in Figure \ref{fig:rk4}.
Taking, e.g., $\lambda=0.21+2.3i$, 
one finds $|R(\lambda)|<1$ but $|R(\lambda/2)|>1$.
Although this example shows that Algorithm \ref{alg:bisection} might formally fail,
it concerns only the trivial case $s=p$ in which there is only one possible choice of
stability polynomial.
We have searched without success for a situation with $s>p$ for which condition
\eqref{h-monotone} is violated.

\subsection{Convergence for starlike regions}
In many important applications the relevant set $\Lambda$ is an infinite set;
for instance, if we wish to design a method for some PDE semi-discretization
that will be stable for any spatial discretization size.
In this case, Problem \ref{prob:stab_opt} is
a semi-infinite program (SIP) as it involves infinitely many constraints.
Furthermore, $\Lambda$ is often a closed curve whose interior is starlike with
respect to the origin; for example, upwind semi-discretizations of hyperbolic
PDEs have this property.  Recall that a region $S$ is starlike if 
$t \in S$ implies $\mu t \in S$ for all $0\le \mu\le 1$.

\begin{lem} \label{lem:starlike}
Let $\Lambda\in\Complex$ be a closed curve passing through the origin and
enclosing a starlike region.
Let $r(h,\Lambda)$ denote the solution of Problem \ref{prob:unconstrained}.
Then condition \eqref{h-monotone} holds.
\end{lem}
\begin{proof}
Let $\Lambda$ be as stated in the lemma.
Suppose $r(h_0,\Lambda)=0$ for some $h_0>0$; then there exists $R(z)$ 
such that $|R(h\lambda)|\le 1$ for all $\lambda \in \Lambda$.
According to the maximum principle, the stability region of $R(z)$ must contain
the region enclosed by $\Lambda$.
Choose $h$ such that $0\le h\le h_0$; then $h\Lambda$ lies in the region enclosed
by $\Lambda$, so $|R(h\lambda)|\le 1$ for $\lambda\in\Lambda$.
\end{proof}

The proof of Lemma \ref{lem:starlike} relies crucially on $\Lambda$ being
an infinite set, but in practice we numerically solve Problem \ref{prob:unconstrained}
with only finitely many constraints.  To this end we
introduce a sequence of discretizations $\Lambda_n$ with
the following properties:
\begin{enumerate}
    \item $\Lambda_n\subset\Lambda$ 
    \item $n_1\le n_2 \implies \Lambda_{n_1} \subset \Lambda_{n_2}$
    \item $\lim_{n\to\infty} \Lambda_n = \Lambda$
    \item $\lim_{n\to\infty}\nu_n=0$ where $\nu_n$ denotes the maximum distance
            from a point in $\Lambda$ to the set $\Lambda_n$:
$$\nu_n = \max_{\gamma\in\Lambda} \min_{\lambda\in\Lambda_n} |\gamma-\lambda|.$$
\end{enumerate}
For instance, $\Lambda_n$ can be taken as an equispaced (in terms of
arc-length, say) sampling of $n$ points.

By modifying Algorithm \ref{alg:bisection}, we can approximate
the solution of the semi-infinite programming problem for starlike regions
to arbitrary accuracy.  At each step we solve Problem \ref{prob:unconstrained}
with $\Lambda_n$ replacing $\Lambda$.
The key to the modified algorithm is to only increase $\hmin$ after obtaining
a certificate of feasibility.  This is done by using the Lipschitz constant
of $R(z)$ over a domain including $h\Lambda$ (denoted by $L(R,h\Lambda)$)
to ensure that $|R(h\Lambda)|\le 1$.
The modified algorithm is stated as Algorithm \ref{alg:bisection-SIP}.

\begin{algorithm}\caption{Bisection for SIP}
\label{alg:bisection-SIP}
\begin{algorithmic}

\State $\hmin = 0$
\State $\hmax = 2s^2/\max|\lambda|$
\State $n=n_0$
\While {$\hmax-\hmin>\epsilon$}  
    \State $h = (\hmax+\hmin)/2$                   \Comment{Bisect}
    \State Solve Problem \ref{prob:unconstrained}
    \If {$r(h,\Lambda_n)<0$ and $\nu_n < -2r/L(R,h\Lambda)$}      \Comment{Certifies that $r(h,\Lambda)<0$}
            \State $\hmin=h$
    \ElsIf {$r(h,\Lambda_n)>0$}                   \Comment{Certifies that $r(h,\Lambda)>0$}
        \State $\hmax=h$
    \Else                                          \Comment{$-\delta<r(h,\Lambda_n)\le0$}
            \State $n \gets 2n$                    \Comment{Reduce the discretization spacing}
    \EndIf
\EndWhile \\
\Return $H_\epsilon = \hmin$
\end{algorithmic}
\end{algorithm}

The following lemma, which characterizes the behavior of Algorithm
\ref{alg:bisection-SIP}, holds whether or not the interior of $\Lambda$ is
starlike.
\begin{lem} \label{lem:termination}
Let $h^{[k]}$ denote the value of $h$ after $k$ iterations of the loop
in Algorithm \ref{alg:bisection-SIP}.  Then either
\begin{itemize}
    \item Algorithm \ref{alg:bisection-SIP} terminates after a finite time with outputs satisfying
        $r(\hmin,\Lambda)\le 0$, $r(\hmax,\Lambda)>0$; or
    \item there exists $j<\infty$ such that $r(h^{[j]},\Lambda)=0$ and
            $h^{[k]}=h^{[j]}$ for all $j\ge k$.
\end{itemize}
\end{lem}
\begin{proof}
First suppose that $r(h^{[j]},\Lambda)=0$ for some $j$.  Then neither 
feasibility nor infeasibility can be certified for this value of $h$,
so $h^{[k]}=h^{[j]}$ for all $j\ge k$.

On the other hand, suppose that $r(h^{[k]},\Lambda)\ne0$ for all $k$.
The algorithm will terminate as long as, for each $h^{[k]}$, 
either feasibility or infeasibility can be certified for large enough $n$.
If $r(h^{[k]},\Lambda)>0$, then necessarily $r(h^{[k]},\Lambda_n)>0$ for large
enough $n$, so infeasibility will be certified.
We will show that if $r(h^{[k]},\Lambda)<0$, then for large enough $n$ the condition
\begin{align} \label{eq:cert}
\nu_n < -2r/L(R,h\Lambda)
\end{align}
must be satisfied.  Since $r(h,\Lambda_n)\le r(h,\Lambda)$
is bounded away from zero and $\lim_{n\to\infty}\nu_n=0$, \eqref{eq:cert} must
be satisfied for large enough $n$ unless the Lipschitz constant $L(R,h\Lambda)$
is unbounded (with with respect to $n$) for some fixed $h$.  Suppose by way
of contradiction that
this is the case, and let $R^{[1]},R^{[2]},\dots$ denote the corresponding sequence of
optimal polynomials.  Then the norm of the vector of coefficients $a_j^{[i]}$ appearing
in $R^{[i]}$ must also grow
without bound as $i\to\infty$.  By Lemma \ref{lem:lipschitz}, this implies that
$|R^{[i]}(z)|$ is unbounded
except for at most $s$ points $z\in\Complex$.  But this contradicts the
condition $|R^{[i]}(h\lambda)|\le1$ for $\lambda\in\Lambda_n$ when $n>s$.
Thus, for large enough $n$ we must have $\nu_n<-2r/L(R,h\Lambda$).
\end{proof}

In practical application, $r(h,\Lambda)=0$ will not be 
detected, due to numerical errors; see Section \ref{sec:threshold}.  
For this reason, in the next theorem
we simply assume that Algorithm \ref{alg:bisection-SIP} terminates.
We also require the following technical result, whose proof is deferred to
the appendix.

\begin{lem} \label{lem:lipschitz}
Let $R^{[1]}, R^{[2]}, \dots$ be a sequence of polynomials of
degree at most $s$ ($s\in\mathbb{N}$ fixed) and denote the coefficients of $R^{[i]}$ by
$a_j^{[i]}\in\mathbb{C}$ ($i\in\mathbb{N}$, $0\le j\le s$): \[R^{[i]}(z) = \sum_{j=0}^s a_j^{[i]} z^j, \quad z\in\mathbb{C}.\] Further, let
$a^{[i]} := (a_0^{[i]},a_1^{[i]},\dots,a_s^{[i]})^T$ and
suppose that the sequence $\|a^{[i]}\|$  is unbounded in $\mathbb{R}$.
Then the sequences $R^{[i]}(z)$ are unbounded for all but at most $s$ points
$z\in\mathbb{C}$.
\end{lem}
\begin{proof} Suppose to the contrary there are $s+1$ $\mathit{distinct}$
complex numbers, say, $z_0, z_1, \ldots, z_s$ such that the vectors 
$r_i:=(R^{[i]}(z_0),R^{[i]}(z_1),\ldots,R^{[i]}(z_s))^T$ ($i\in\mathbb{N}$) are \textit{bounded} in $
\mathbb{C}^{s+1}$.
Let $V$ denote the $(s+1)\times(s+1)$ Vandermonde matrix whose $k^\mathrm{th}$ row
($0\le k\le s+1)$ is $(1,z_k,z_k^2,\ldots,z_k^s)$. Then $V$ is invertible and  
we have $a^{[i]}=V^{-1}r_i$ ($i\in\mathbb{N}$), so if $\left|\!\left|\!\left|\cdot\right|\!\right|\!\right|$ denotes the induced matrix norm, then
\[
\|a^{[i]}\|=\| V^{-1}r_i\|\le \left|\!\left|\!\left|V^{-1}\right|\!\right|\!\right|\,\|r_i\|.
\]
But, by assumption, the right hand side is bounded, whereas the left hand side is not.
\end{proof}

\begin{theorem}[Global convergence for strictly starlike regions]
Let $\Lambda$ be a closed curve that encloses a region that is starlike with respect
to the origin.  Suppose that Algorithm \ref{alg:bisection-SIP} terminates
for all small enough $\epsilon$, and let $H_\epsilon$ denote the value returned by
Algorithm \ref{alg:bisection-SIP} for a given $\epsilon$.  Let
$\Hopt$ denote the solution of Problem \ref{prob:stab_opt}.
Then
$$\lim_{\epsilon\to0} H_\epsilon = \Hopt.$$
\end{theorem}

\begin{proof}
Due to the assumptions and Lemma \ref{lem:termination}, we have that
$r(\hmin,\Lambda)<0<r(\hmax,\Lambda)$.
Then Lemma \ref{lem:starlike} implies that $\hmin<\Hopt<\hmax$.
Noting that also $\hmax-\hmin<\epsilon$, the result follows.
\end{proof}

Despite the lack of a general global convergence proof, Algorithm
\ref{alg:bisection} works very well in practice even for general $\Lambda$ when $p>1$.
In all cases we have tested and for which the true $\Hopt$ is known (see
Section \ref{sec:examples}), Algorithm \ref{alg:bisection}
appears to converge to the globally optimal solution.
Furthermore, Algorithm \ref{alg:bisection} is very fast.  For these
reasons, we consider the (much slower) Algorithm \ref{alg:bisection-SIP}
to be of primarily theoretical interest, and we base our practical
implementation on Algorithm \ref{alg:bisection}.

\section{Numerical implementation\label{sec:implementation}}
We have made a prototype implementation of Algorithm \ref{alg:bisection} in
\matlab.
The implementation relies heavily on the \cvx package
\cite{CVX,gb08}, a \matlab-based modeling system for convex optimization, which
in turn relies on the interior-point solvers \sedumi \cite{sturm1999using} 
and \sdpt \cite{tutuncu2003solving}.
The least deviation problem (Problem
\ref{prob:unconstrained}) can be succinctly stated in four lines of the
\cvx problem language, and for many cases is solved in under a second by either
of the core solvers.

Our implementation re-attempts failed solves (see Section \ref{change-basis})
with the alternate interfaced solver.  In our test cases, we observed that the
\sdpt interior-point solver was slower, but more robust than \sedumi.
Consequently, our prototype implementation uses \sdpt by default.  

Using the resulting implementation, we were able to successfully solve problems to within 
$0.1\%$ accuracy or better with scaled eigenvalue
magnitudes $|h\lambda|$ as large as 4000.  As an example, comparing with
results of \cite{Bogatyrev2005a} for spectra on the real axis with $p=3, s=27,$ 
our results are accurate to 6 significant digits.

\subsection{Feasibility threshold\label{sec:threshold}}

In practice, CVX often returns a small positive objective ($r\approx 10^{-7}$)
for values of $h$ that are just feasible.  Hence
the bisection step is accepted if $r<\epsilon$ where $\epsilon\ll 1$.
The results are generally insensitive (up to the first few digits) to the choice of $\epsilon$
over a large range of values; we have used $\epsilon=10^{-7}$ for all results in this work.
The accuracy that can be achieved is eventually limited by the 
need to choose a suitable value $\epsilon$.

\subsection{Conditioning and change of basis\label{change-basis}}
Unfortunately, for large values of $h\lambda$, the numerical solution of
Problem \ref{prob:unconstrained}
becomes difficult due to ill-conditioning of the constraint matrix.  
Observe from \eqref{eq:polyform} that the constrained quantities
$R(h\lambda)$ are related to the decision variables $a_j$ through
multiplication by a Vandermonde matrix.  Vandermonde matrices are 
known to be ill-conditioned for most choices of abscissas.
For very large $h\lambda$, the resulting \cvx problem cannot be reliably
solved by either of the core solvers.

A first approach to reducing the condition number of the constraint matrix is to
rescale the monomial basis.  We have found that a more robust approach for
many types of spectra can be obtained by choosing a basis that is approximately
orthogonal over the given spectrum $\{\Lambda\}$.
Thus we seek a solution of the form
\begin{align}
    R(z) & = \sum_{j=0}^s a_j Q_j(z) & \mbox{where } \quad \quad \quad
    Q_j(z) & = \sum_{k=0}^j b_{jk} z^k.
\end{align}
Here $Q_j(z)$ is a degree-$j$ polynomial chosen to give a well-conditioned constraint
matrix.  The drawback of not using the monomial basis is that the dimension of the problem
is $s+1$ (rather than $s+1-p$) and we must now impose the order conditions explicitly:
\begin{align}
      \sum_{j=0}^s a_j b_{jk} = \frac{1}{k!} \quad \quad \mbox{ for } \ \ k=0,1,\dots,p.
\end{align}
Consequently, using a non-monomial basis increases the number of design
variables in the problem and introduces an equality constraint matrix $B \in \Real^{p \times s}$
that is relatively small (when $p\ll s$), but usually very poorly conditioned.
However, it can dramatically improve the conditioning of the inequality
constraints.

The choice of the basis $Q_j(z)$
is a challenging problem in general.
In the special case of a negative real spectrum, an obvious choice is
the Chebyshev polynomials (of the first kind) $T_j$, shifted
and scaled to the domain $[hx, 0]$ where  $x=\min_{\lambda\in\Lambda} \operatorname{Re}(\lambda)$,
via an affine map:
\begin{align} \label{eq:shifted_cheb}
    Q_j(z) & = T_j\left(1+\frac{2z}{hx}\right).
\end{align}
The motivation for using this basis 
is that $|Q_j(h\lambda)|\le 1$ for all $\lambda\in[hx,0]$.
This basis is also suggested by the fact that $Q_j(z)$ is the optimal stability
polynomial in terms of negative real axis inclusion for $p=1,s=j$.  In Section
\ref{sec:examples}, we will see that this choice of basis works well for more
general spectra when the largest magnitude eigenvalues lie near the negative
real axis.

As an example, we consider a spectrum of 3200 equally spaced values $\lambda$ in the 
interval $[-1,0]$.  The exact solution is known to be  $h=2s^2$.
Figure \ref{fig:cond} shows the relative error as well as the inequality constraint matrix condition
number obtained by using the monomial \eqref{eq:polyform} and Chebyshev
\eqref{eq:shifted_cheb} bases.  Typically, the
solver is accurate until the condition number reaches about $10^{16}$.  This supports
the hypothesis that it is the conditioning of the inequality constraint matrix that
leads to failure of the solver.  The Chebyshev basis keeps the condition number small
and yields accurate answers even for very large values of $h$.

\begin{figure}
  \center
  \includegraphics[width=5in]{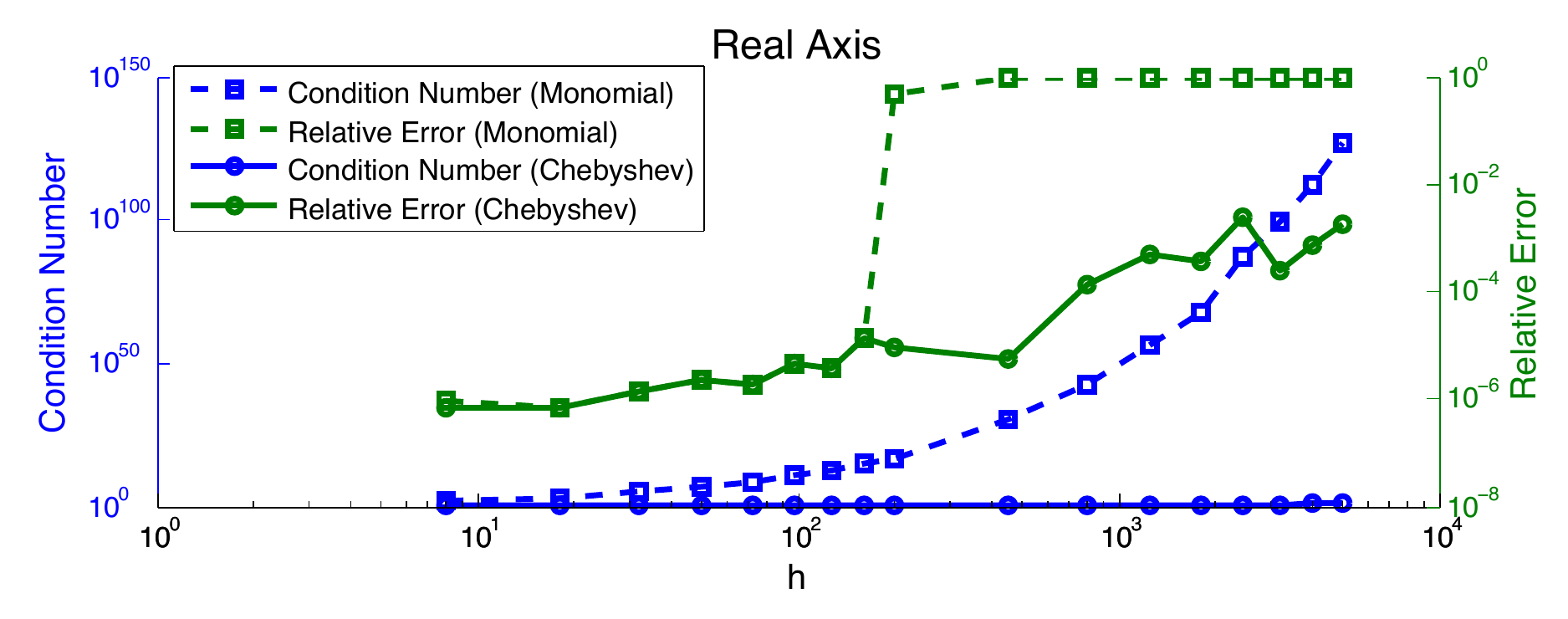}
  \caption{Condition number of principal constraint matrix and relative solution accuracy versus optimal step size.  The points along a given curve correspond to different choices of $s$.\label{fig:cond}}
\end{figure}


\subsection{Choice of initial upper bound}
The bisection algorithm requires as input an initial $\hmax$ such that $r(\hmax,\Lambda)>0$.
Theoretical values can be obtained using the classical upper bound of $2s^2/x$ if $\Lambda$
encloses a negative real interval $[x,0]$, or using the upper bound given in
\cite{Sanz-Serna1986a} if $\Lambda$ encloses an ellipse in the left half-plane.
Alternatively, one could start with a guess and successively double it until
$r(\hmax,\Lambda)>0$ is satisfied.  Since evaluation of $r(h,\Lambda)$ is 
typically quite fast, finding a tight initial $\hmax$ is not an essential
concern.

\section{Examples\label{sec:examples}}
We now demonstrate the effectiveness of our algorithm by applying it to 
determine optimally stable polynomials (i.e., solve Problem \ref{prob:stab_opt})
for various types of spectra.  As stated above, we use Algorithm \ref{alg:bisection}
for its simplicity, speed, and effectiveness.  When $\Lambda$ corresponds to
an infinite set, we approximate it by a fine discretization.

\subsection{Verification}
In this section, we apply our algorithm to some well-studied cases
with known exact or approximate results in order to verify its accuracy
and correctness.  In addition to the real axis, imaginary axis, and disk
cases below, we have successfully recovered the results of \cite{Niegemann2011}.
Our algorithm succeeds in finding the globally optimal
solution in every case for which it is known, except in some cases of 
extremely large step sizes for which the underlying solvers (\sdpt and \sedumi)
eventually fail.  

\subsubsection{Negative real axis inclusion}
Here we consider the largest $h$ such that $[-h,0]\in S$ by taking $\Lambda=[-1,0]$.
This is the most heavily studied case in the literature, 
as it applies to the semi-discretization of parabolic PDEs
and a large increase of $\Hopt$ is possible when $s$ is increased
(see, e.g., \cite{Riha1972a,vanderhouwen1996,Medovikov1998,Bogatyrev2005a,Skvortsov2011a}).
For first-order accurate methods ($p=1$), the optimal polynomials are just
shifted Chebyshev polynomials, and the optimal timestep is $\Hopt=2s^2$.
Many special analytical and numerical techniques have been
developed for this case; the most powerful
seems to be that of Bogatyrev \cite{Bogatyrev2005a}.

We apply our algorithm to a discretization of $\Lambda$ (using 6400 evenly-spaced points)
and using the shifted and scaled Chebyshev basis \eqref{eq:shifted_cheb}.
Results for up to $s=40$ are shown in Table \ref{tbl:realaxis123410}
(note that we list $\Hopt/s^2$ for easy comparison, since $\Hopt$ is approximately
proportional to $s^2$ in this case).  We include results for $p=10$ to
demonstrate the algorithm's ability to handle high-order methods.
For $p=1$ and $2$, the values computed here match those available in the literature
\cite{vanderHouwen1972}.  Most of the values for $p=3,4$ and $10$ are new results.
Figure \ref{fig:realstab} shows some examples of stability regions for optimal methods.
As observed in the literature, it seems that $\Hopt/s^2$ tends to a constant
(that depends only on $p$) as $s$ increases.  For large values of $s$, some results
in the table have an error of about $10^{-3}$ due to inaccuracies in the
numerical results provided by the interior point solvers.

\input{tables/realaxis-123410}

\begin{figure}
  \center
  \subfigure[$p=4, s=20$]{\includegraphics[height=0.8in]{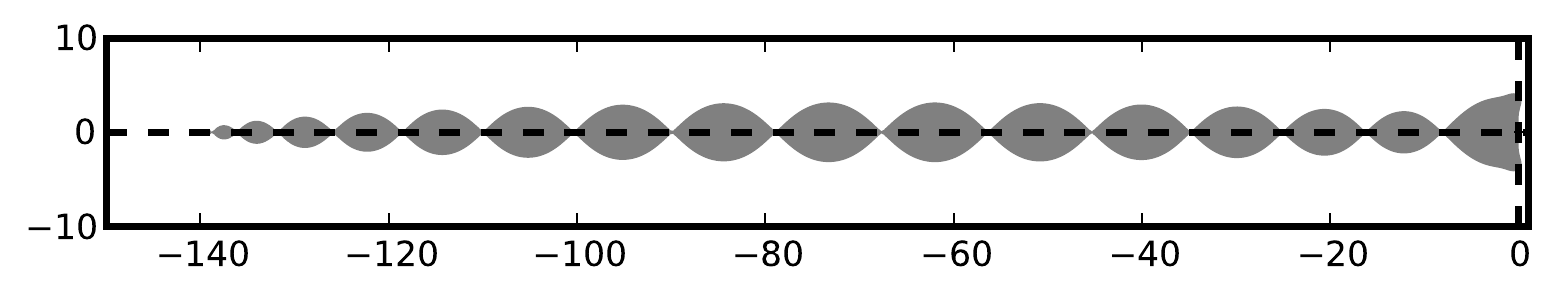}} \\
  \subfigure[$p=10, s=20$]{\includegraphics[height=0.8in]{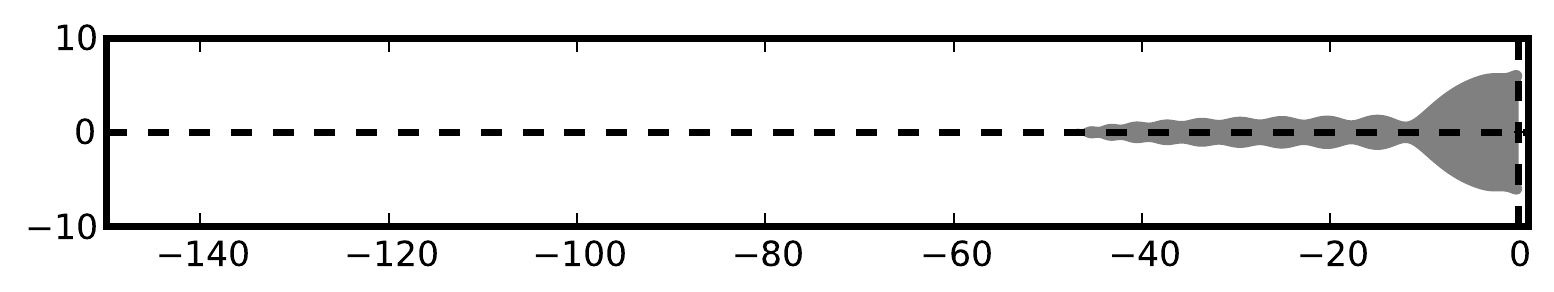}}
  \caption{Stability regions of some optimal methods for real axis inclusion.\label{fig:realstab}}
\end{figure}

\subsubsection{Imaginary axis inclusion}
Next we consider the largest $h$ such that $[-ih,ih]\in S$ by taking
$\Lambda=xi, x\in[-1,1]$.
Optimal polynomials for imaginary axis inclusion have also been studied by many
authors, and a number of exact results are known or conjectured
\cite{vanderHouwen1972,vichnevetsky1983,kinnmark1984b,kinnmark1984a,ingemar1986,vanderhouwen1996}.
We again approximate the problem, taking $N=3200$ evenly-spaced values 
in the interval $[0,i]$ (note that stability regions are necessarily symmetric
about the real axis since $R(z)$ has real coefficients).  We use a ``rotated"
Chebyshev basis defined by
$$
Q_j(z) = i^j T_j\left(\frac{iz}{hx}\right),
$$
where $x=\max_i(|\operatorname{Im}(\lambda_i)|)$.  Like the Chebyshev basis for
the negative real axis, this basis dramatically improves the robustness of the algorithm for 
imaginary spectra.
Table \ref{tbl:imagaxis1234} shows the optimal effective step sizes.
In agreement with \cite{vanderHouwen1972,kinnmark1984a}, we find $H=s-1$ for $p=1$ (all $s$)
and for $p=2$ ($s$ odd).  
We also find $H=s-1$ for $p=1$ and $s$ even, which was conjectured in \cite{vichnevetsky1983}
and confirmed in \cite{vanderhouwen1996}.
We find $\Hopt=\sqrt{s(s-2)}$ for $p=2$ and $s$ even, strongly suggesting
that the polynomials given in \cite{kinnmark1984b} are optimal for these cases;
on the other hand, our results show that those polynomials, while third order
accurate, are not optimal for $p=3$ and $s$ odd.
Figure \ref{fig:imagstab} shows some examples of stability regions for optimal methods.

\input{tables/imagaxis-1234}

\begin{figure}
  \center
  \subfigure[$p=1, s=7$]{\includegraphics[width=1.5in]{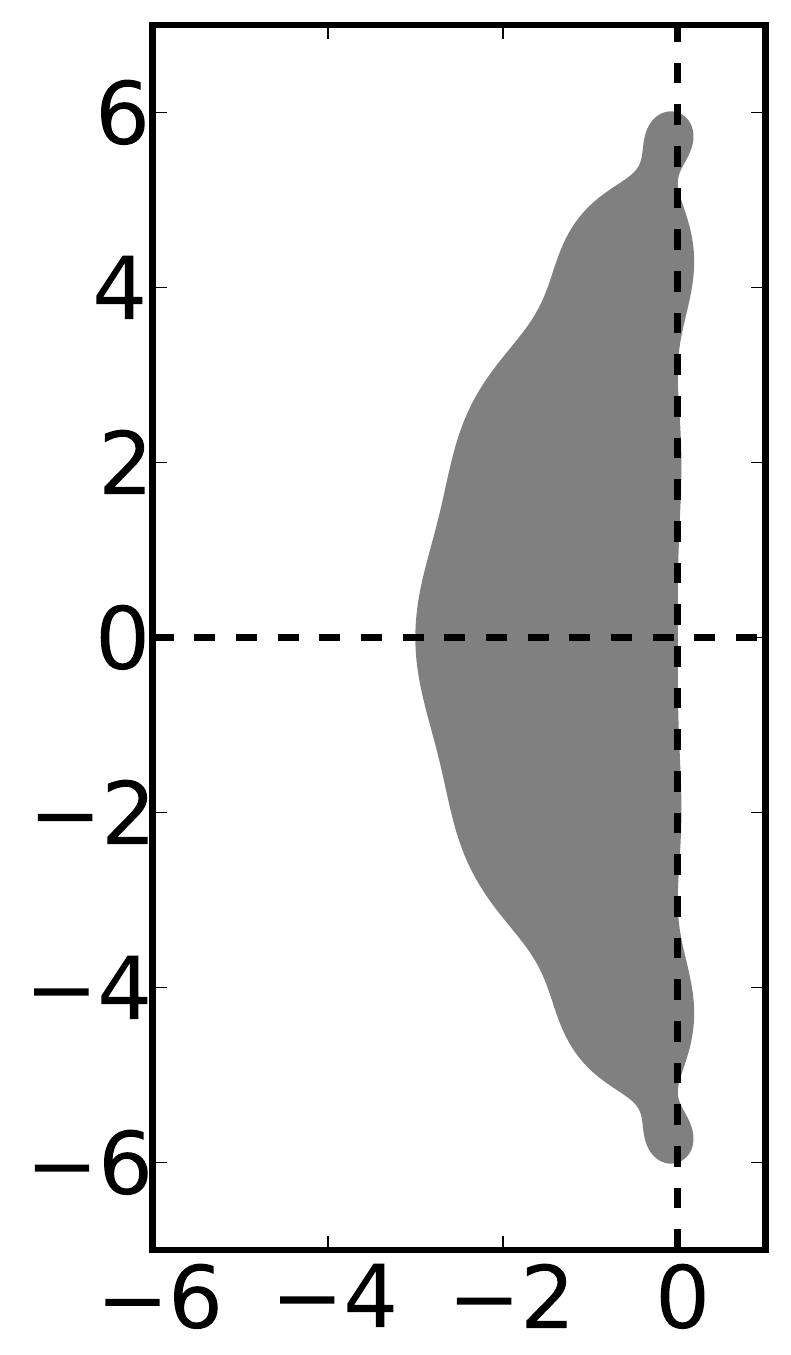}} \quad \quad
  \subfigure[$p=4, s=7$]{\includegraphics[width=1.5in]{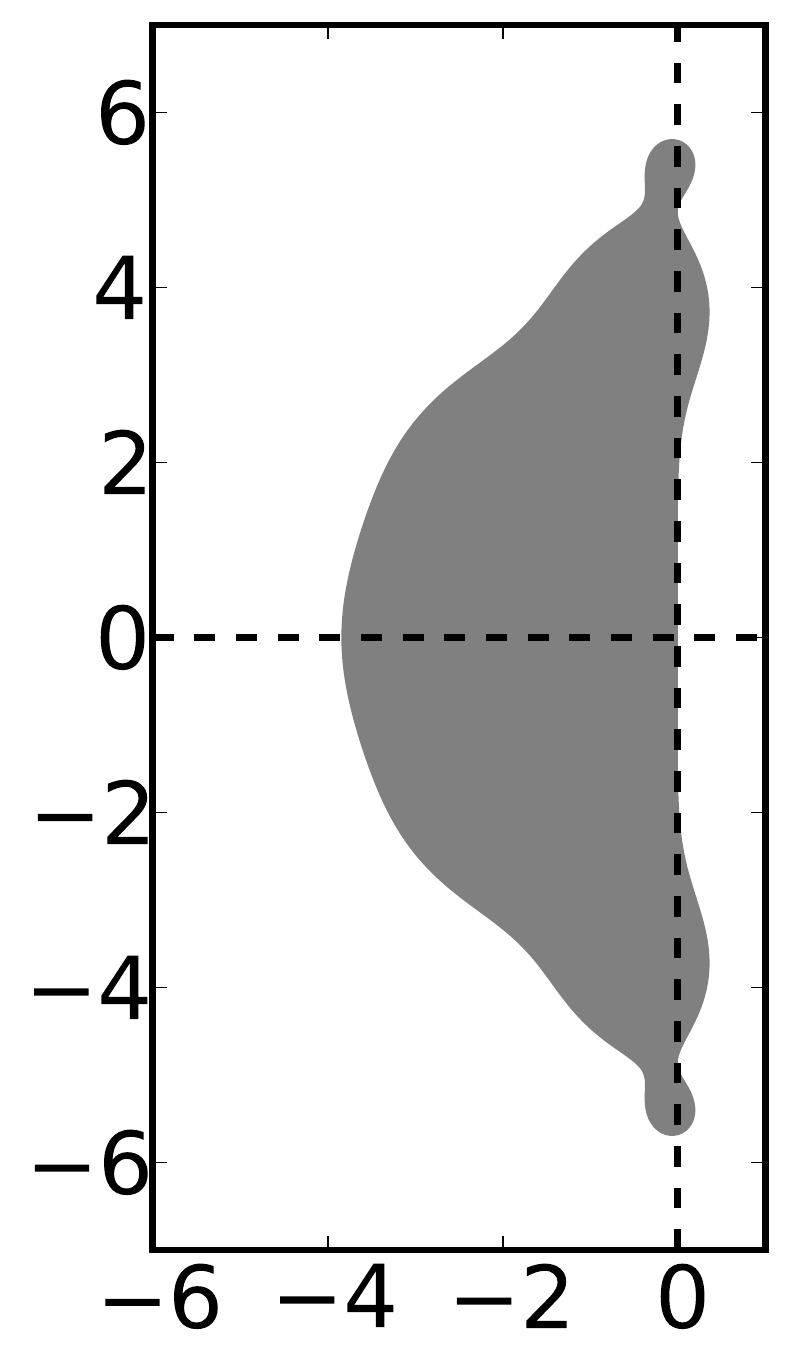}}
  \caption{Stability regions of some optimal methods for imaginary axis inclusion.\label{fig:imagstab}}
\end{figure}

\subsubsection{Disk inclusion}
In the literature, attention has been paid to stability regions that include
the disk 
\begin{align} \label{eq:disk}
D(h) & = \{z : |1+z/h|\le1\},
\end{align}
for the largest possible $h$.
As far as we know, the optimal result for $p=1$ ($\Hopt=s$) was
first proved in \cite{jeltsch1978largest}.
The optimal result for $p=2$ ($\Hopt=s-1$) was first proved
in \cite{vichnevetsky1983}.  Both results have been unwittingly rediscovered
by later authors.  For $p>2$, no exact results are available.

We use the basis
$$ Q_j(z) = \left(1+\frac{z}{h}\right)^j.$$
Note that $Q_j(z)$ is the optimal polynomial for the case
$s=j$, $p=1$.  This basis can also be motivated by recalling that Vandermonde
matrices are perfectly conditioned when the points involved are equally spaced on the
unit circle.  Our basis can be obtained by taking the monomial basis and applying an
affine transformation that shifts the unit circle to the disk \eqref{eq:disk}.
This basis greatly improves the robustness of the algorithm for this particular spectrum.
We show results for $p\le 4$ in Figure \ref{fig:disk}.
For $p=3$ and $s=5,6$, our results give a small improvement over those
of \cite{Jeltsch2006a}.  Some examples of optimal stability regions
are plotted in Figure \ref{fig:diskstab}.

\begin{figure}
  \center
  \includegraphics[width=3in]{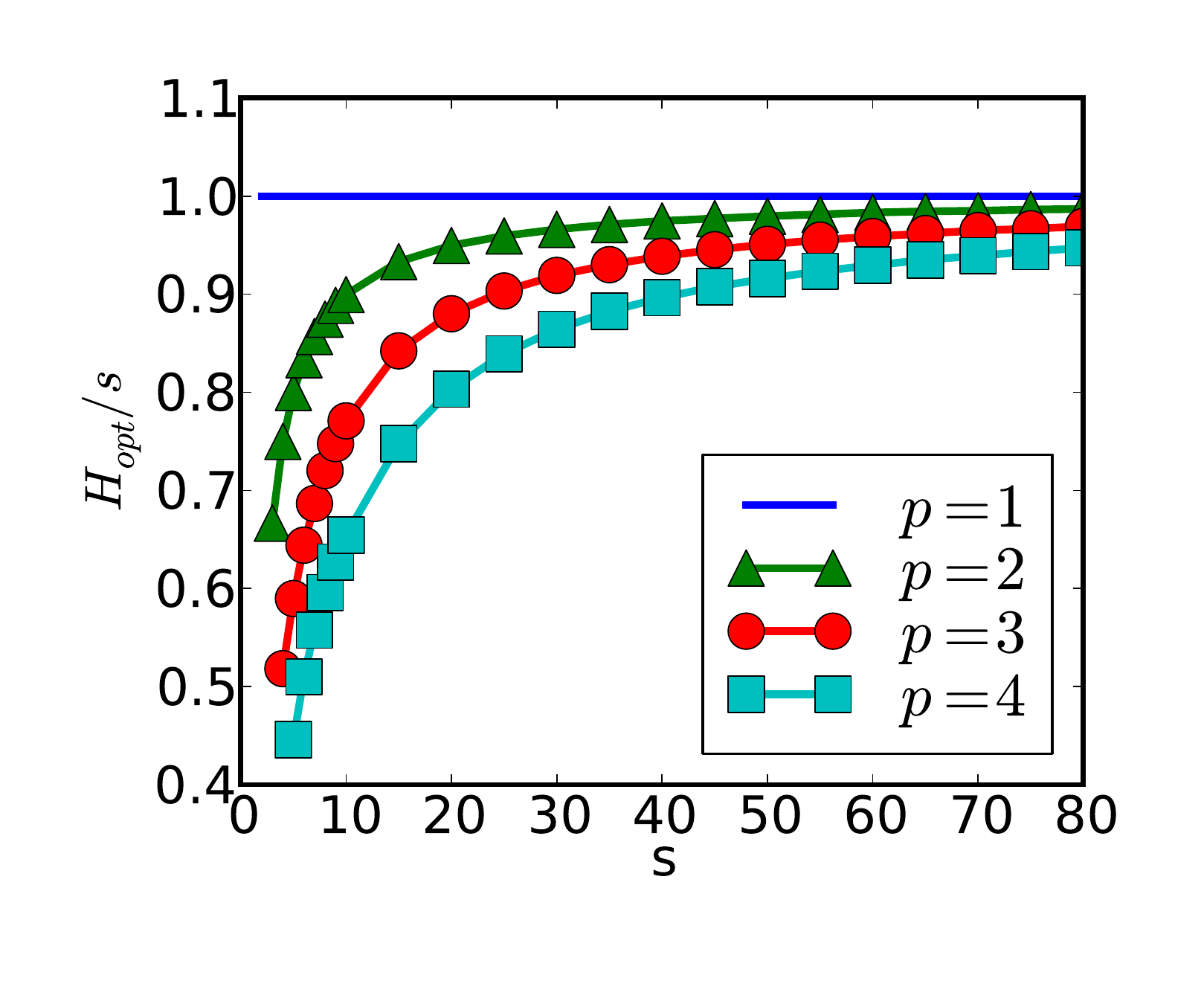}
  \caption{Relative size of largest disk that can be included in the stability region
            (scaled by the number of stages).\label{fig:disk}}
\end{figure}

\begin{figure}
  \center
  \subfigure[$p=3, s=8$]{\includegraphics[width=2.5in]{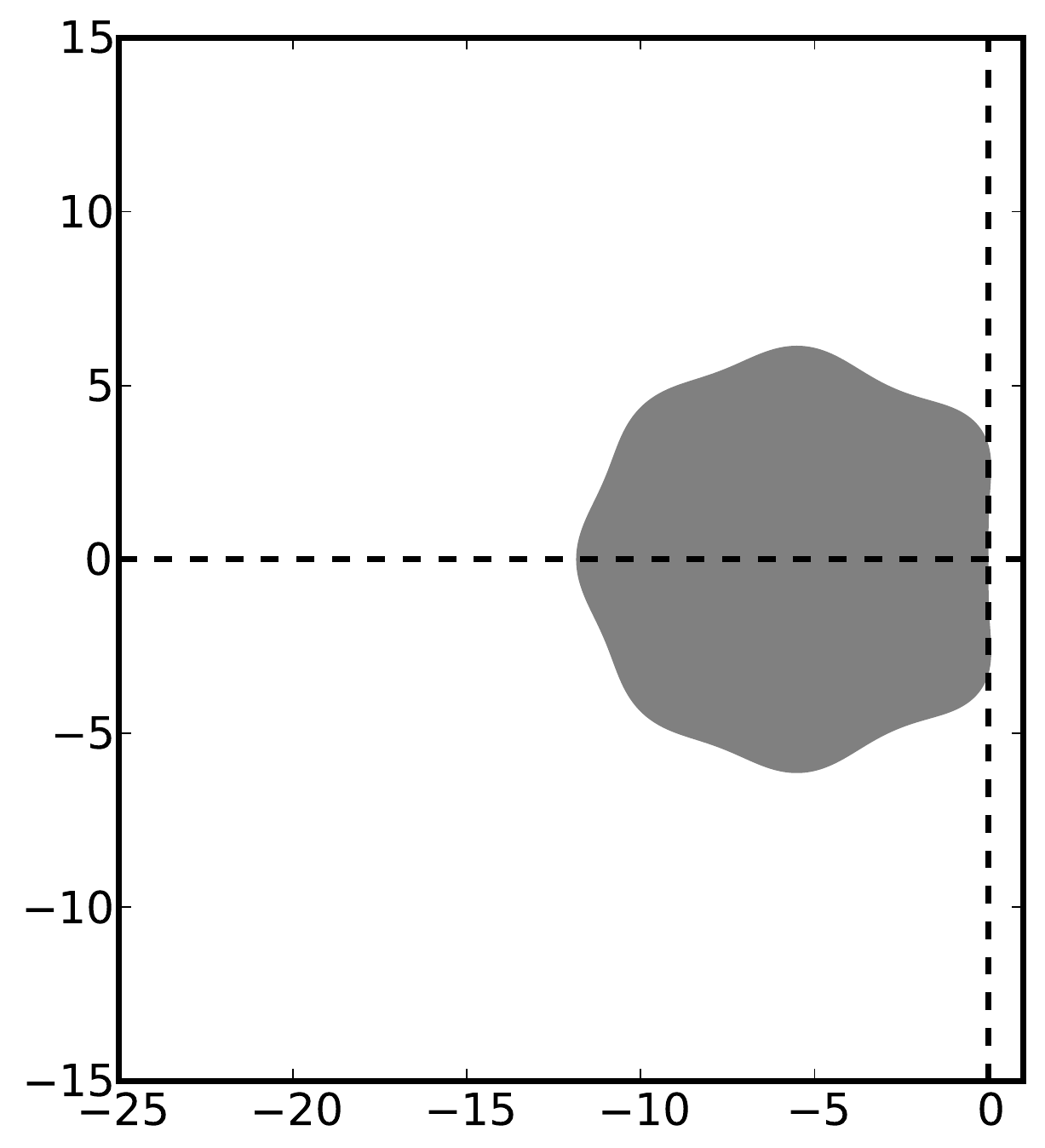}}
  \subfigure[$p=4, s=15$]{\includegraphics[width=2.5in]{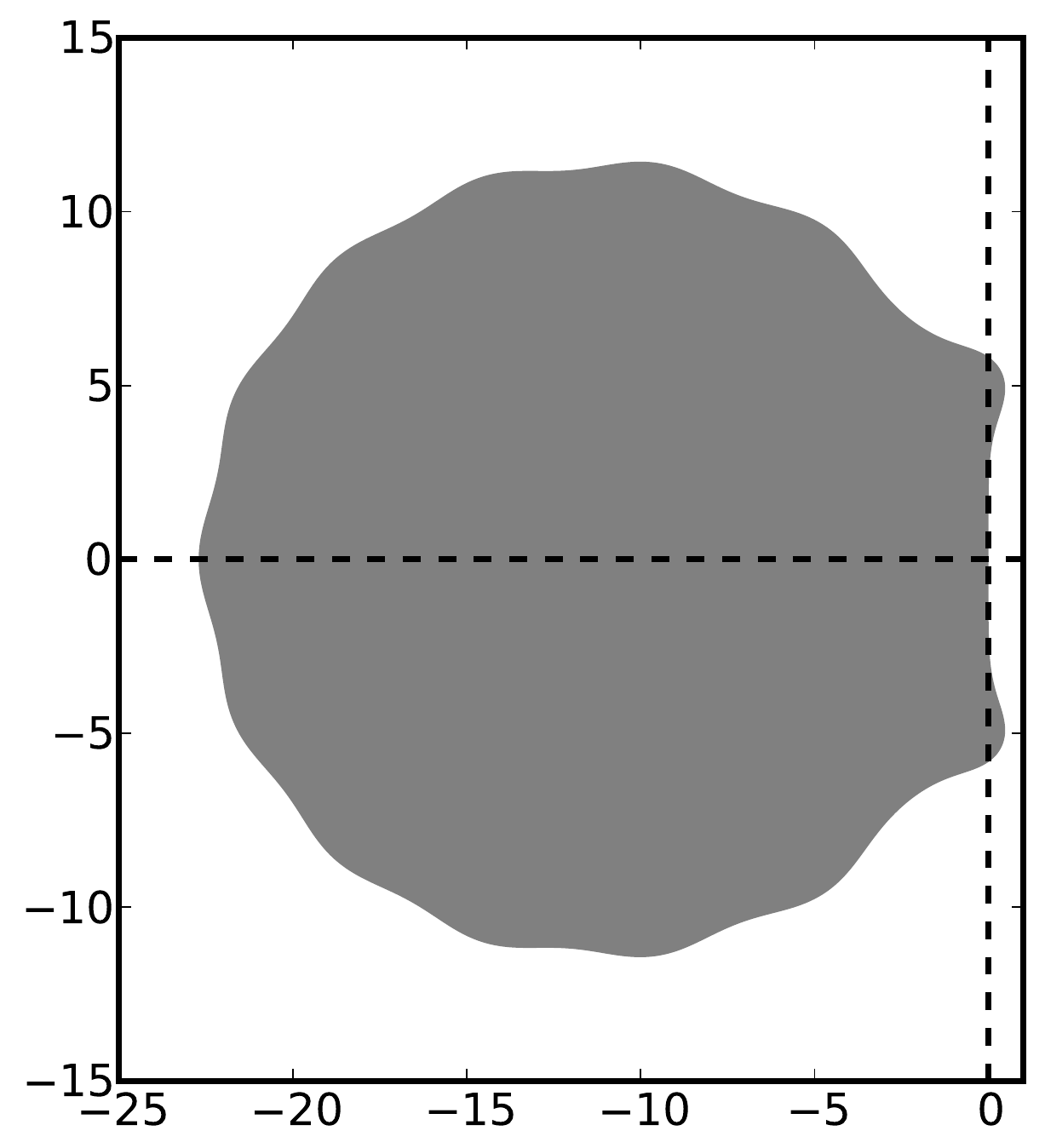}}
  \caption{Stability regions of some optimal methods for disk inclusion.\label{fig:diskstab}}
\end{figure}


%

\subsection{Spectrum with a gap\label{sec:gap}}
We now demonstrate the effectiveness of our method for more general spectra.
First we consider the case of a dissipative problem with two time
scales, one much faster than the other.  This type of problem was
the motivation for the development of projective integrators in \cite{Gear2003a}.
Following the ideas outlined there we consider
\begin{align} \label{eq:gap}
\Lambda & = \{z : |z|=1, \Real(z)\le 0\} \cup \{z : |z-\alpha|=1\}.
\end{align}
We take $\alpha=20$ and use the shifted and scaled Chebyshev basis
\eqref{eq:shifted_cheb}.  Results are shown in Figure \ref{fig:gap}.  A
dramatic increase in efficiency is achieved by adding a few extra stages.  

\begin{figure}
  \center
  \subfigure[Optimal effective step size]{\includegraphics[height=1.8in]{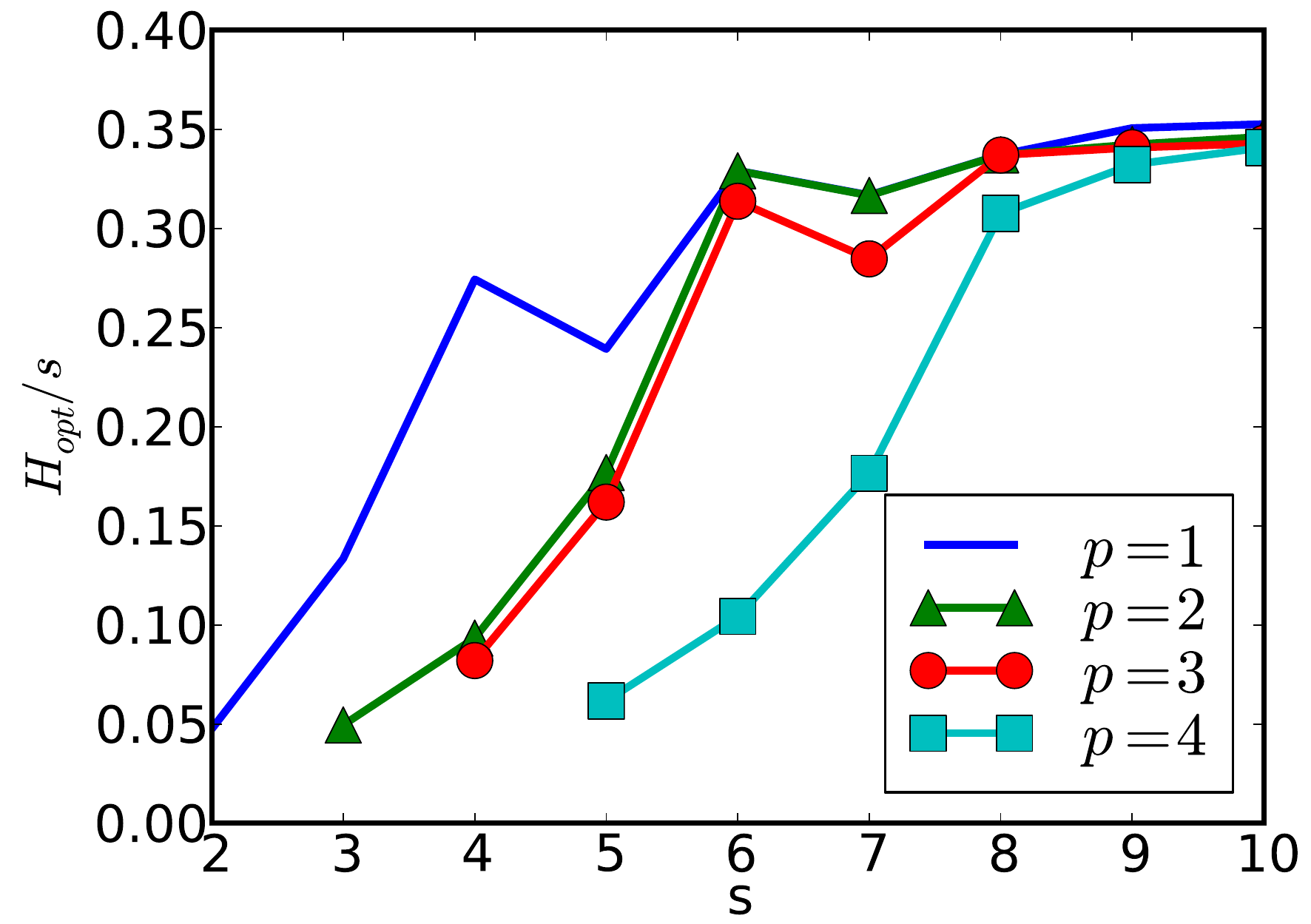}}
  \subfigure[Optimal stability region for $p=1, s=6, \alpha=20$ (stable step size $\approx 1.975$)]{\includegraphics[height=1.8in]{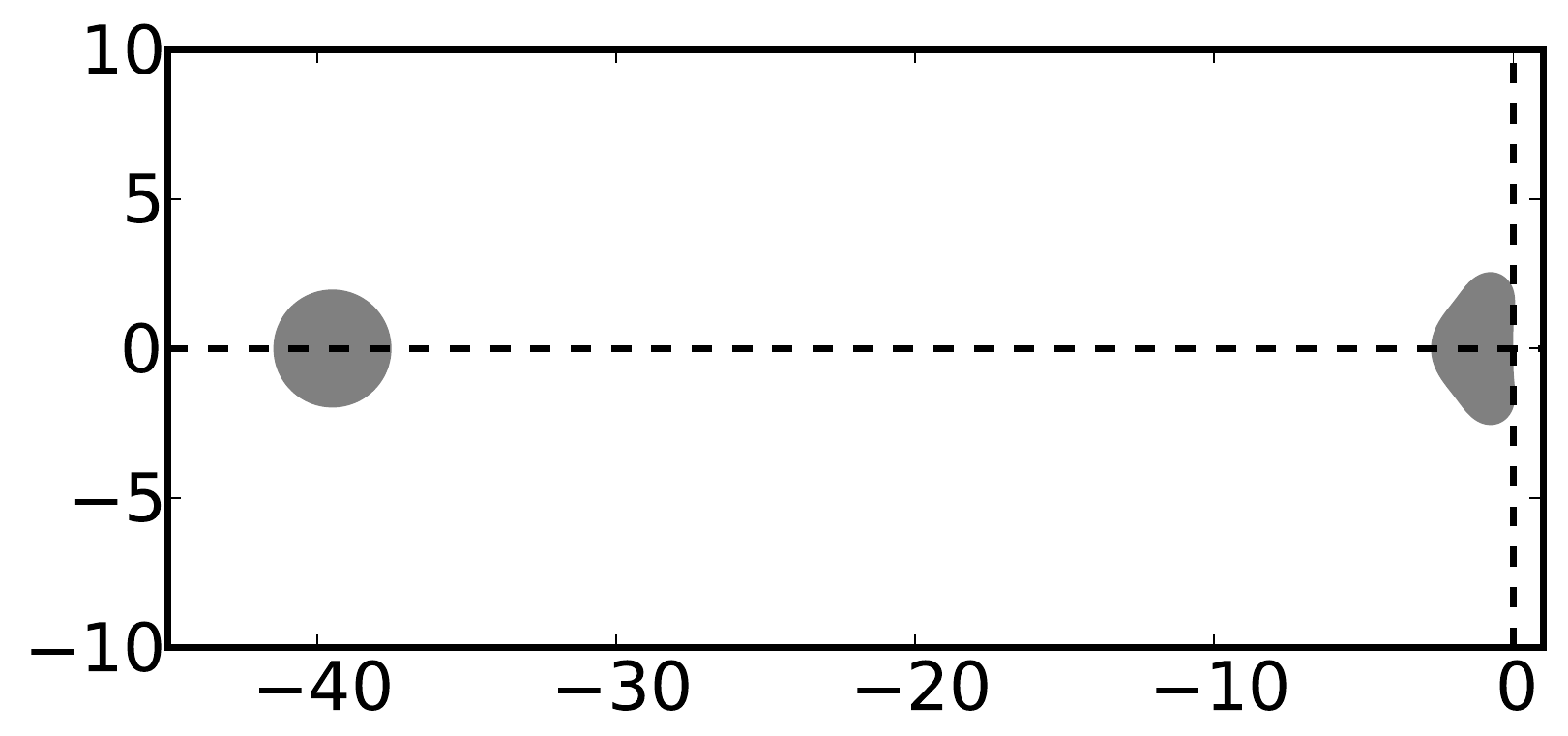}}
  \caption{Optimal methods for spectrum with a gap \eqref{eq:gap} with $\alpha=20$.\label{fig:gap}}
\end{figure}

\subsection{Legendre pseudospectral discretization\label{sec:pseudospectrum}}
Next we consider a system obtained from semidiscretization of the advection equation
on the interval $[-1,1]$ with homogeneous Dirichlet boundary condition:
$$u_t = u_x \ \ \ \ \ \ \ \ u(t,x=1)=0.$$
The semi-discretization is based on pseudospectral collocation at points given
by the zeros of the Legendre polynomials; we take $N=50$ points.  The
semi-discrete system takes the form \eqref{eq:ivp}, where $L$ is the Legendre
differentiation matrix, whose eigenvalues are shown in Figure \ref{fig:eig}.
We compute an optimally stable polynomial based on the spectrum of the
matrix, taking $s=7$ and $p=1$.  The stability region of the resulting method is plotted
in Figure \ref{fig:leg_unstable}.  Using an appropriate step size, all
the scaled eigenvalues of $L$ lie in the stability region.  However, this
method is unstable in practice for any positive step size; Figure \ref{fig:unstable_leg}
shows an example of a computed solution after three steps, where the initial
condition is a Gaussian.  The resulting instability is non-modal, meaning that it does
not correspond to any of the eigenvectors of $L$ (compare \cite[Figure 31.2]{trefethen-pseudospectra}).

This discretization is now well-known
as an example of non-normality \cite[Chapters 30-32]{trefethen-pseudospectra}.
Due to the non-normality, it is necessary to consider pseudospectra in 
order to design an appropriate integration scheme.
The $\epsilon$-pseudospectrum (see \cite{trefethen-pseudospectra}) is the set
$$ \{ z\in\Complex : \|(z-D)^{-1}\|>1/\epsilon\}.  $$
The $\epsilon$-pseudospectrum (for $\epsilon=2$) is shown with the eigenvalues
in Figure \ref{fig:eig-pssp}.
The instability observed above occurs because the stability region does not
contain an interval on the imaginary axis about the origin, whereas the
pseudospectrum includes such an interval.

We now compute an optimally stable integrator based on the 2-pseudospectrum.
This pseudospectrum is computed using an approach proposed in \cite[Section 20]{trefethen1999},
with sampling on a fine grid.  In order to reduce the number of constraints and speed up
the solution, we compute the convex hull of the resulting set
and apply our algorithm.  The resulting stability region is
shown in Figure \ref{fig:leg_stable}.  It is remarkably well adapted;
notice the two isolated roots that ensure stability of the modes corresponding to the
extremal imaginary eigenvalues.  We have verified that this method produces a stable
solution, in agreement with theory (see Chapter 32 of \cite{trefethen-pseudospectra}); Figure \ref{fig:stable_leg}
shows an example of a solution computed with this method.  The initial Gaussian pulse advects
to the left.

\begin{figure}
  \center
  \subfigure[Eigenvalues.\label{fig:eig}]{\includegraphics[width=1.8in]{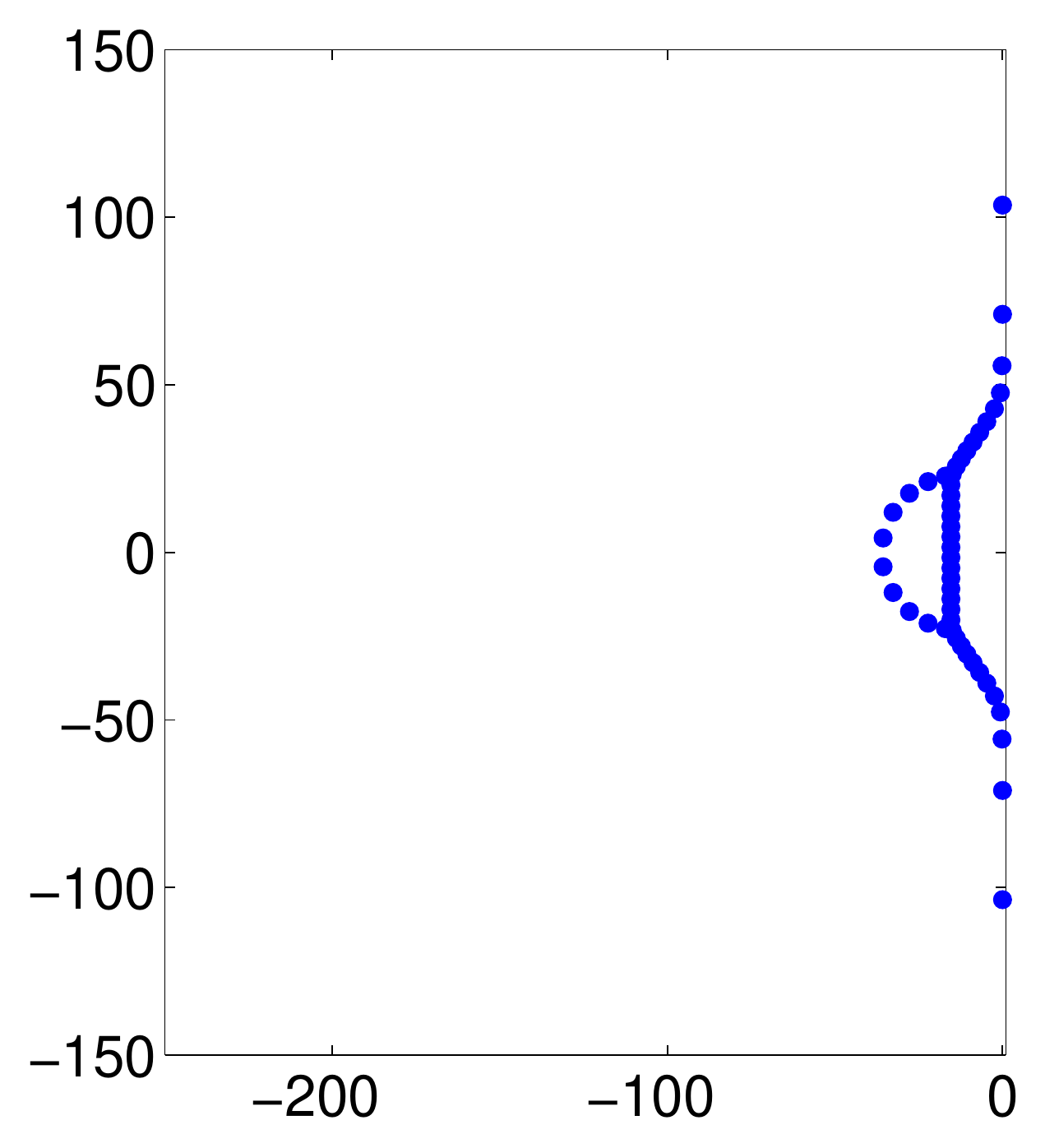}} \ \  \ \ \
  \subfigure[Eigenvalues and pseudospectrum (the boundary of the 2-pseudospectrum is plotted).\label{fig:eig-pssp}]{\includegraphics[width=1.8in]{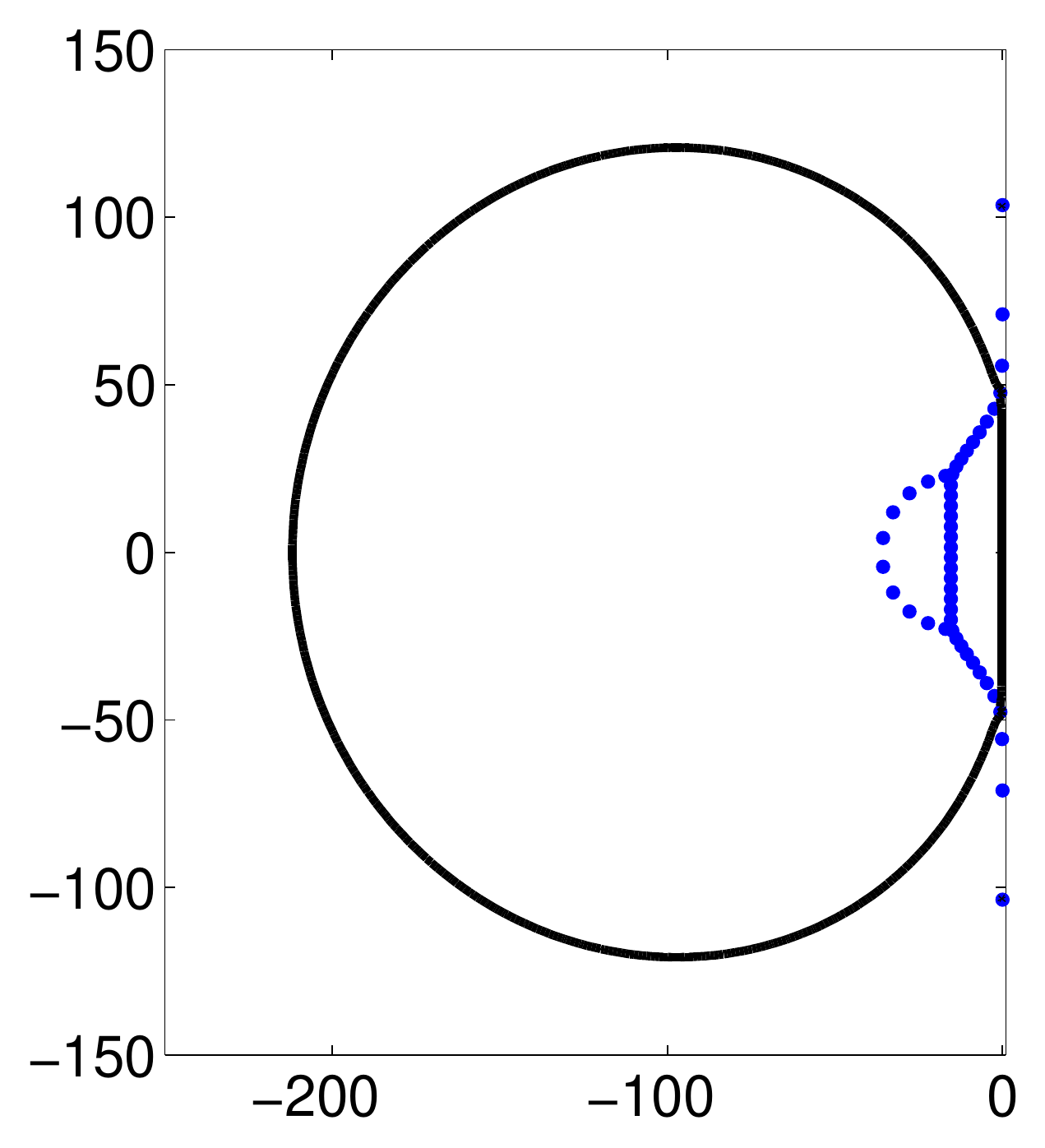}} \\
  \subfigure[Optimized stability region based on eigenvalues.\label{fig:leg_unstable}]{\includegraphics[height=2.in]{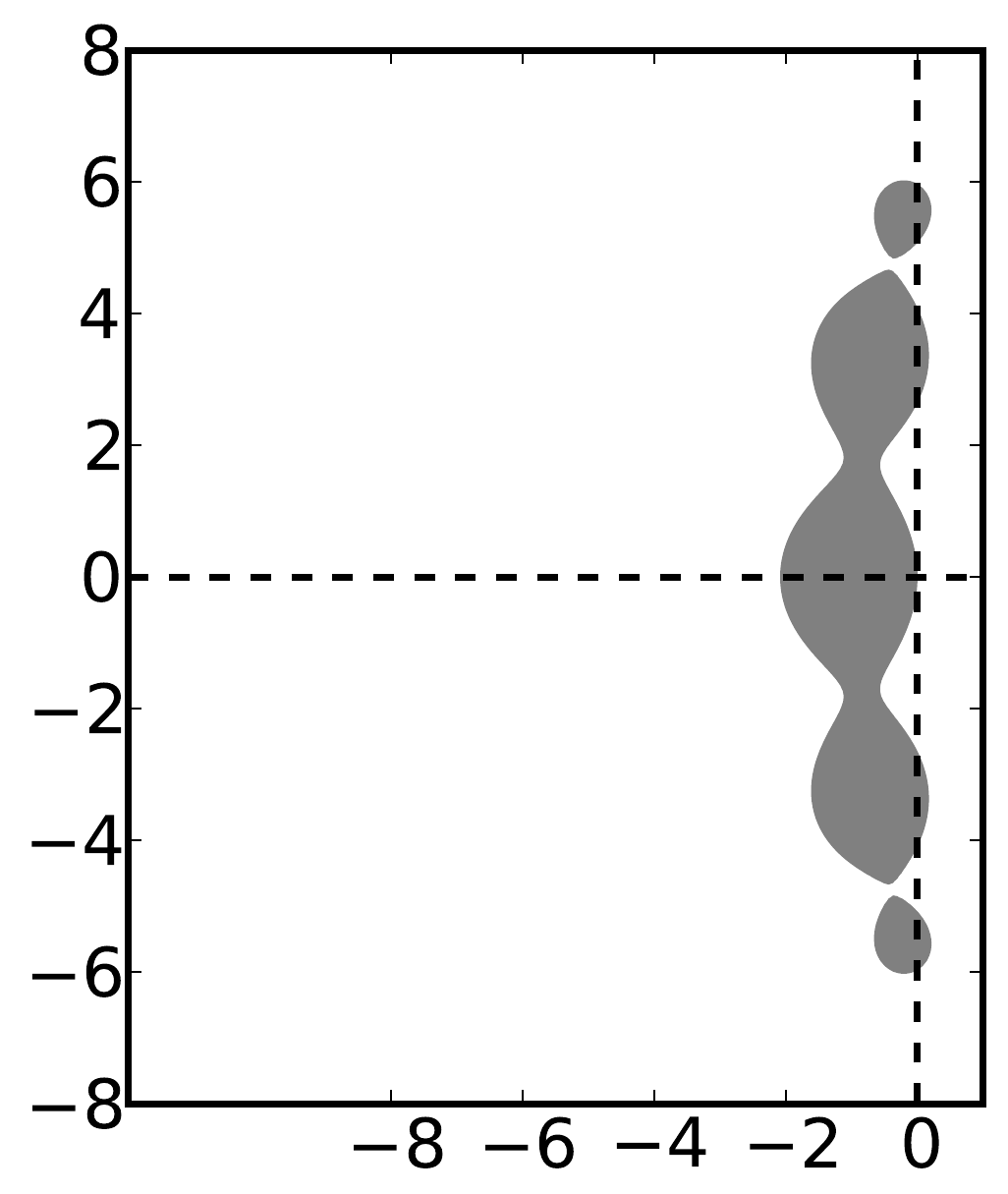}} \ \ \ \ \  \quad 
  \subfigure[Optimized stability region based on pseudospectrum.\label{fig:leg_stable}]{\includegraphics[height=2.in]{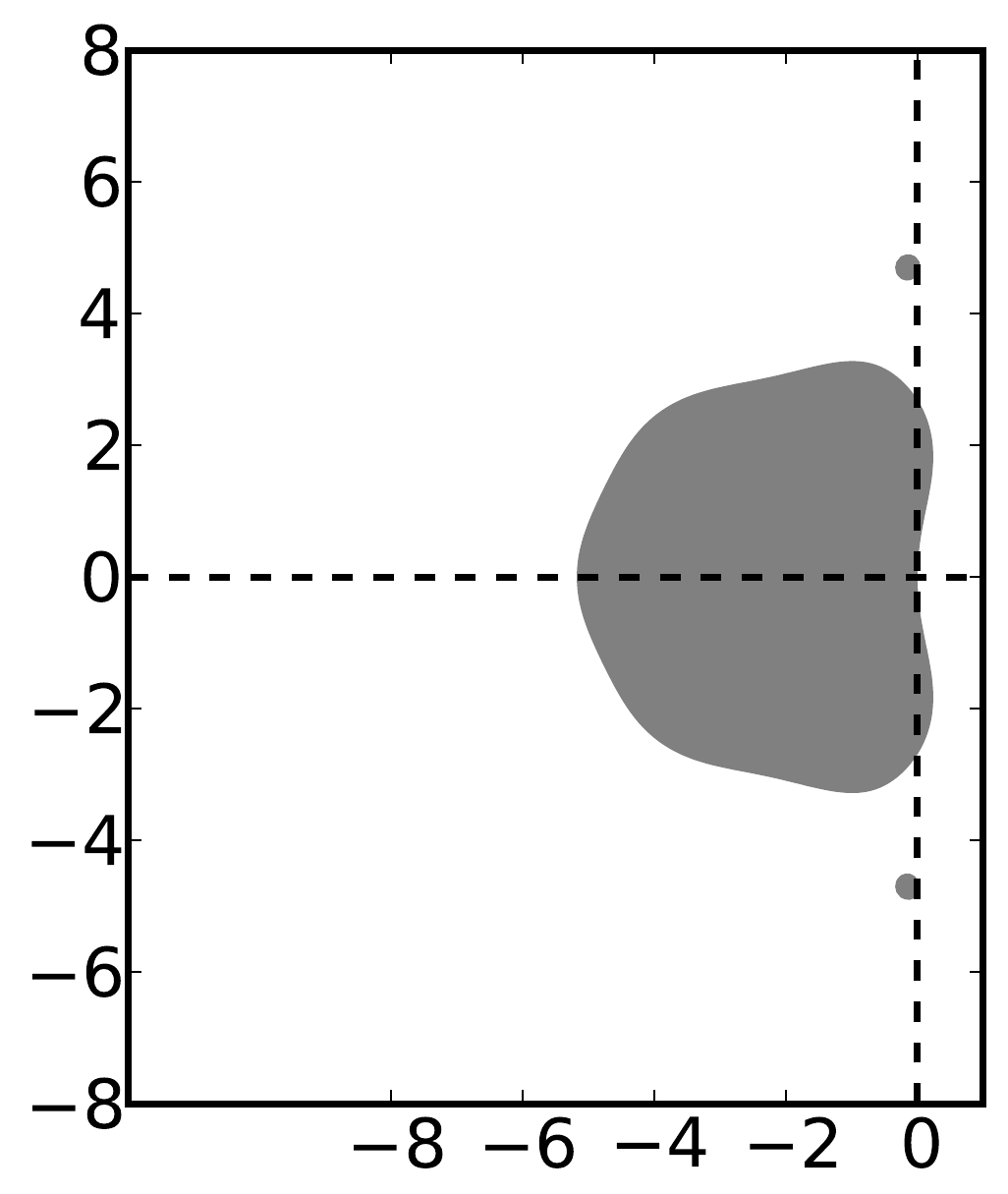}} \\
  \subfigure[Solution computed with method based on spectrum.\label{fig:unstable_leg}]{\includegraphics[height=1.5in]{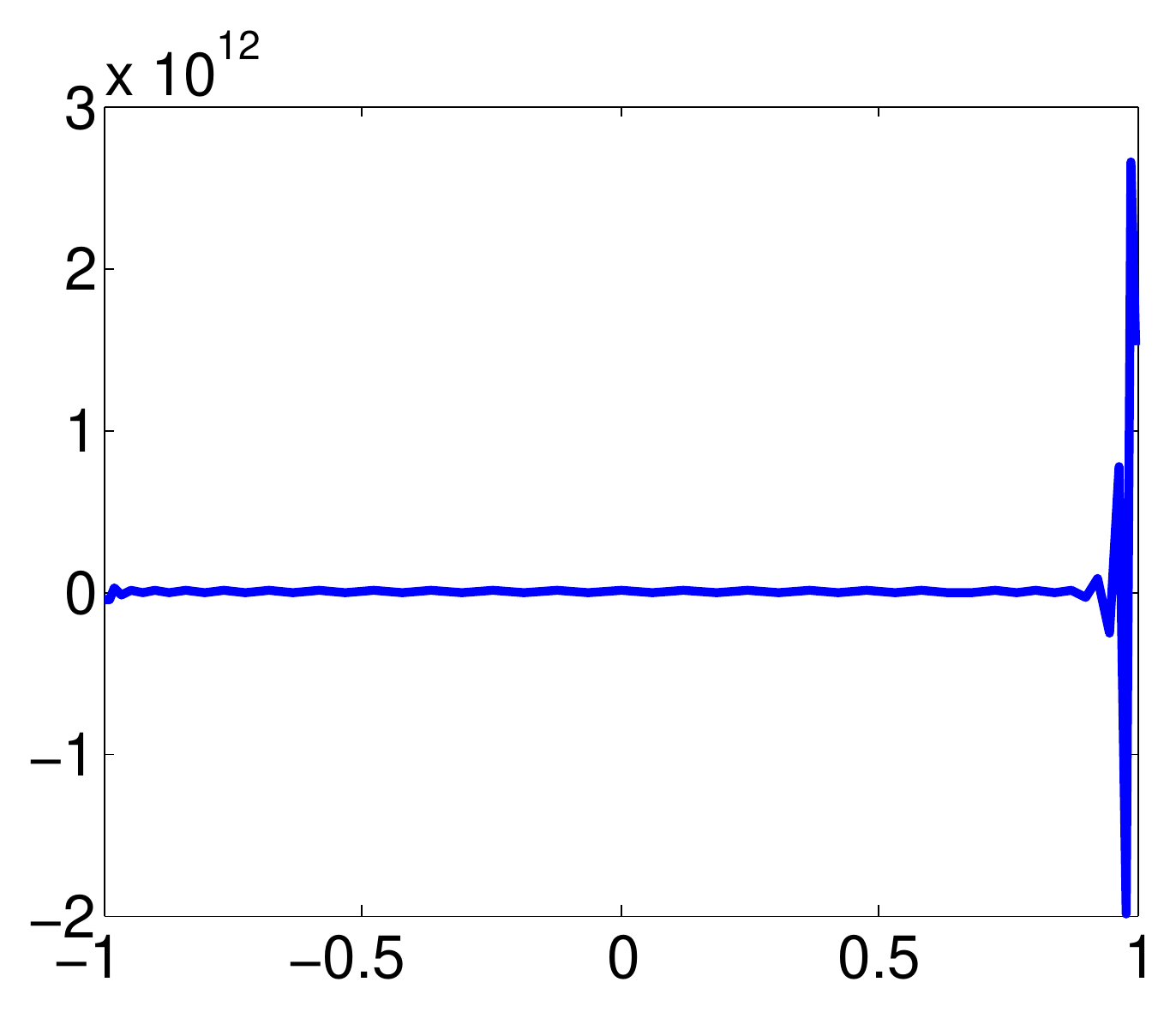}} \ \ \ \ \  \quad 
  \subfigure[Solution computed with method based on pseudospectrum.\label{fig:stable_leg}]{\includegraphics[height=1.5in]{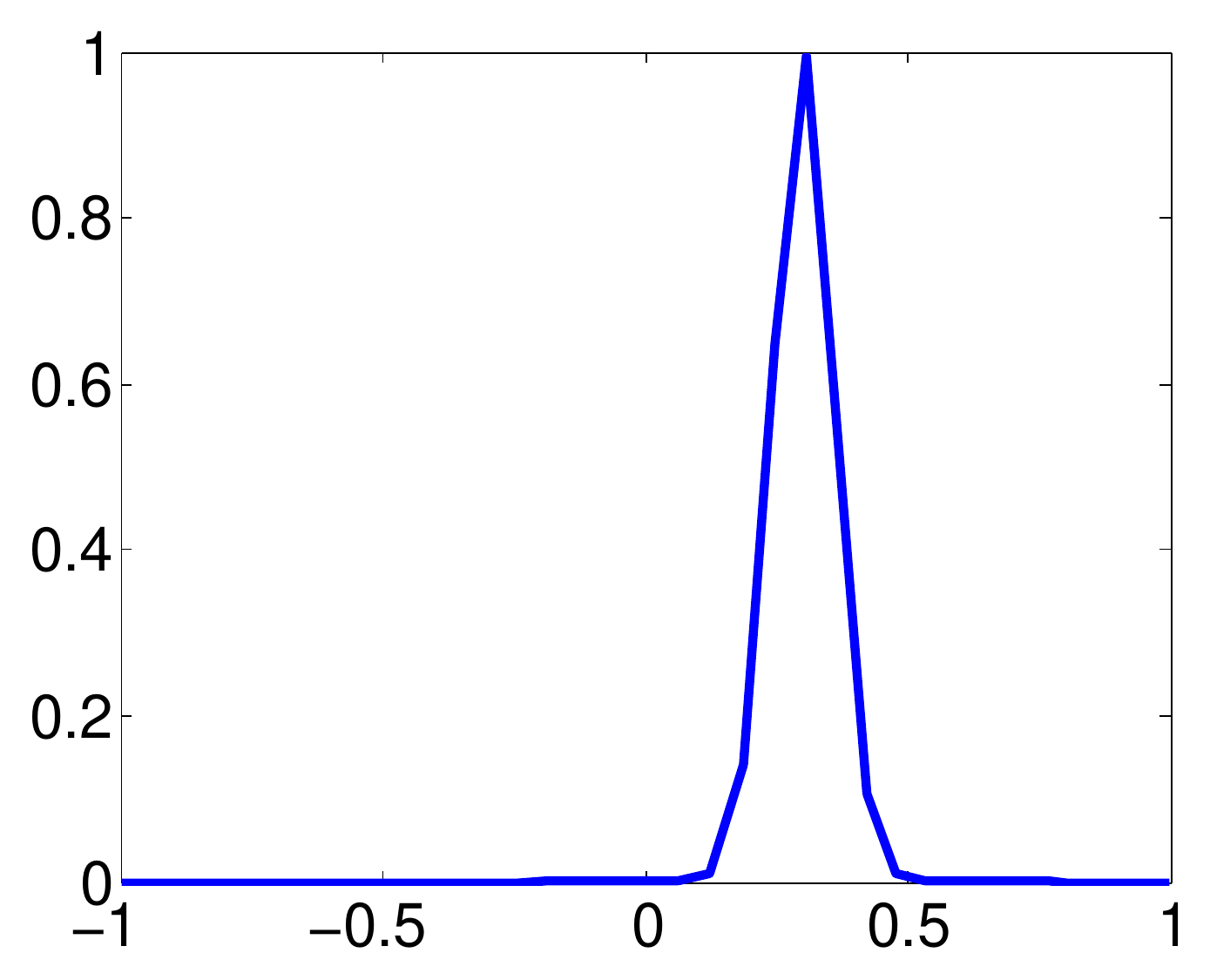}} \\
  \caption{Results for the Legendre differentiation matrix with $N=50$.\label{fig:legendre}}
\end{figure}

\subsection{Thin rectangles\label{sec:rectangles}}
A major application of explicit Runge--Kutta methods with many stages is 
the solution of moderately stiff advection-reaction-diffusion problems
\cite{Hundsdorfer_Verwer_2003,verwer2004rkc}.
For such problems, the stability region must include not only a large interval on the
negative real axis, but also some region around it, due to convective terms.
If centered differences are used for the advective terms, it is natural to require
that a small interval on the imaginary axis be included.  Hence, one may be interested
in methods that contain a rectangular region
\begin{align}
\Lambda_\kappa = \{\lambda\in\Complex : -\beta\le\imag(\lambda)\le\beta, \ \
-\kappa\le \real(\lambda)\le 0\}.
\end{align}
for given $\kappa, \beta$.  No methods optimized for such regions appear in 
the literature, and the available approaches for devising methods with extended real
axis stability (including those of \cite{torrilhon2007essentially}) cannot
be applied to such regions.  Because of this, previously
existing methods are applicable only if upwind differencing is applied to convective terms
\cite{verwer2004rkc,torrilhon2007essentially}.  

For this example, rather than parameterizing by the step size $h$, we assume that
a desired step size $h$ and imaginary axis limit $\beta$ are given based on the
convective terms, which generally require small step sizes for accurate
resolution.  We seek to find (for given $s,p$) the polynomial
\eqref{eq:polyform} that includes $\Lambda_\kappa$ for $\kappa$ as large as possible.
This could correspond to selection of an optimal integrator based on the
ratio of convective and diffusive scales (roughly speaking, the Reynolds number).
Since the desired stability region lies relatively near the negative real axis,
we use the shifted and scaled Chebyshev basis \eqref{eq:shifted_cheb}.

Stability regions of some optimal methods are shown in Figure \ref{fig:rectstab}.
The outline of the included rectangle is superimposed in black.
The stability region for $\beta=10, s=20$, shown in Figure \ref{fig:rectstab} is
especially interesting as it is very nearly rectangular.  A closeup view of the
upper boundary is shown in Figure \ref{fig:closeup}.

\begin{figure}
  \center
  \subfigure[$\beta=1$, $p=1, s=10$]{\includegraphics[width=2.5in]{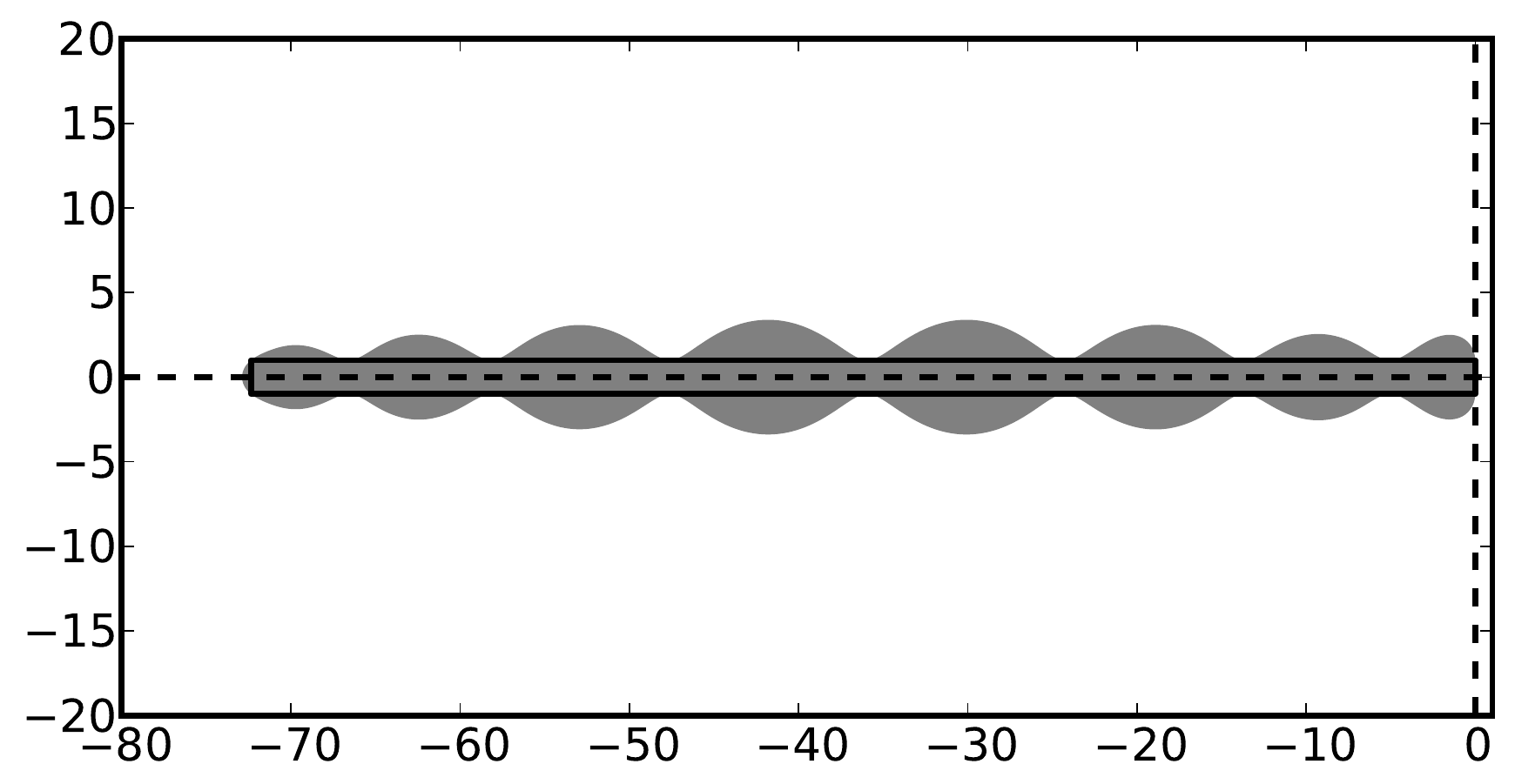}}
  \subfigure[$\beta=10$, $p=1, s=20$\label{fig:rectrect}]{\includegraphics[width=2.5in]{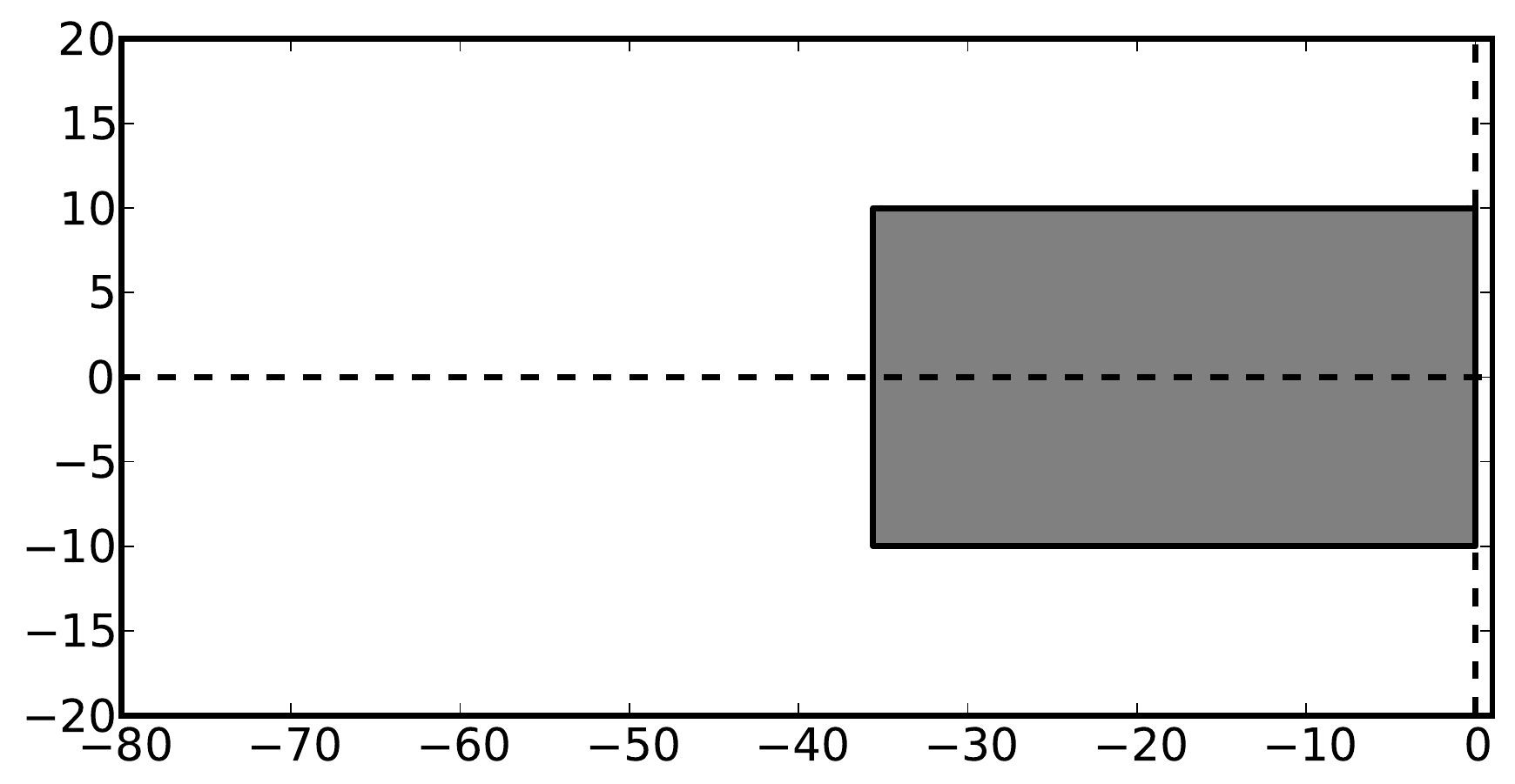}}
  \caption{Stability regions of some optimal methods for thin rectangle inclusion.\label{fig:rectstab}}
\end{figure}

\begin{figure}
  \center
  \includegraphics[width=2.5in]{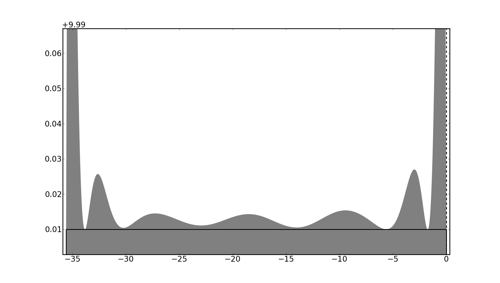}
  \caption{Closeup view of upper boundary of the rectangular stability region plotted in Figure \ref{fig:rectstab}.\label{fig:closeup}}
\end{figure}

%
%
%
%

\section{Discussion}
The approach described here can speed up the integration of IVPs for which
\begin{itemize}
    \item explicit Runge--Kutta methods are appropriate;
    \item the spectrum of the problem is known or can be approximated; and
    \item stability is the limiting factor in choosing the step size.
\end{itemize}
Although we have considered only linear initial value problems,
we expect our approach to be useful in designing integrators for nonlinear
problems via the usual approach of considering the spectrum of the Jacobian.
A first successful application of our approach to nonlinear PDEs appears in
\cite{parsani2012}.

The amount of speedup depends strongly on the spectrum of the problem,
and can range from a few percent to several times or more.
Based on past work and on results presented in Section \ref{sec:examples}, we expect that
the most substantial gains in efficiency will be realized for systems whose
spectra have large negative real parts, such as for semi-discretization of
PDEs with significant diffusive or moderately stiff reaction components.
As demonstrated in Section \ref{sec:examples}, worthwhile improvements may also be attained for general systems,
and especially for systems whose spectrum contains gaps.  

The work presented here suggests several extensions and areas for further
study.
For very high polynomial degree, the convex subproblems required by our
algorithm exhibit poor numerical conditioning.  We have proposed a first
improvement by change of basis, but further improvements in this regard
could increase the robustness and accuracy of the algorithm.
It seems likely that our algorithm exhibits global convergence in
general circumstances beyond those for which we have proven convergence.
The question of why bisection seems to always lead to globally optimal
solutions merits further investigation.
While we have focused primarily on design of the stability properties of a scheme,
the same approach can be used to optimize accuracy efficiency,
which is a focus of future work.
Our algorithm can also be applied in other ways;  for instance, it could be used
to impose a specific desired amount of dissipation for use in multigrid
or as a kind of filtering.


We remark that the problem of determining optimal polynomials subject
to convex constraints is very general.  Convex optimization techniques have already
been exploited to solve similar problems in filter design \cite{Davidson2010},
and will likely find further applications in numerical analysis.

{\bf Acknowledgments.}  We thank Lajos Loczi for providing a simplification
of the proof of Lemma \ref{lem:lipschitz}.  We are grateful to R.J. LeVeque
and L.N. Trefethen for helpful comments on a draft of this work.

\bibliography{stabregopt}

\end{document}

%% file: macros.tex
\newtheorem{theorem}{Theorem}
\newtheorem{cor}{Corollary}
\newtheorem{lem}{Lemma}
\newtheorem{problem}{Problem}
\newtheorem{assumption}{Assumption}

\newtheorem{remark}{Remark}

\newtheorem{dfn}{Definition}

\newcommand{\qh}{\hat{q}}
\newcommand{\be}{\begin{equation}}
\newcommand{\ee}{\end{equation}}
\newcommand{\bq}{\mathbf{q}}
\newcommand{\ba}{\mathbf{a}}
\newcommand{\bx}{\mathbf{x}}
\newcommand{\br}{\mathbf{r}}
\newcommand{\imh}{{i-\frac{1}{2}}}
\newcommand{\imth}{{i-\frac{3}{2}}}
\newcommand{\imfh}{{i-\frac{5}{2}}}
\newcommand{\iph}{{i+\frac{1}{2}}}
\newcommand{\ipth}{{i+\frac{3}{2}}}
\newcommand{\ipfh}{{i+\frac{5}{2}}}
\newcommand{\ipmh}{{i \pm \frac{1}{2}}}
\newcommand{\jph}{{j+\frac{1}{2}}}
\newcommand{\jmh}{{j-\frac{1}{2}}}
\newcommand{\Aop}{{\cal A}}
\newcommand{\Bop}{{\cal B}}
\newcommand{\Wop}{{\cal W}}
\newcommand{\Oop}{{\cal O}}
\newcommand{\DQ}{\Delta Q}
\newcommand{\Dq}{\Delta q}
\newcommand{\Dx}{\Delta x}
\newcommand{\Dy}{\Delta y}
\newcommand{\Du}{\Delta u}
\newcommand{\bu}{\mathbf{u}}
\newcommand{\bv}{\mathbf{v}}
\newcommand{\bg}{\mathbf{g}}
\newcommand{\bw}{\mathbf{w}}
\newcommand{\bU}{\mathbf{U}}
\newcommand{\bV}{\mathbf{V}}
\newcommand{\bF}{\mathbf{F}}
\newcommand{\Lop}{{\cal L}}
\newcommand{\Cop}{{\cal C}}
\newcommand{\Fop}{{\cal F}}
\newcommand{\Dofr}{{\cal D}(r)}
\newcommand{\Dt}{\Delta t}
\newcommand{\DtFE}{\Delta t_\textup{FE}}
\newcommand{\bbA}{\mathbf{A}}
\newcommand{\bbZ}{\mathbf{Z}}
\newcommand{\bbK}{\mathbf{K}}
\newcommand{\bM}{\mathbf{M}}
\newcommand{\bbI}{\mathbf{I}}
\newcommand{\bbb}{\mathbf{b}}
\newcommand{\bB}{\mathbf{B}}
\newcommand{\bR}{\mathbf{R}}
\newcommand{\bbe}{\mathbf{e}}
\newcommand{\bbone}{\mathbf{1}}

\newcommand{\dx}{\Delta x}
\newcommand{\dt}{\Delta t}
\newcommand{\hfp}{\hat{f}_{j+\half}}
\newcommand{\hfn}{\hat{f}_{j-\half}}
\newcommand{\aik}{\alpha_{i,k}}
\newcommand{\bik}{\beta_{i,k}}
\newcommand{\lt}{\tilde{L}}

\newcommand{\hf}{\frac{1}{2}}
\newcommand{\fracStrut}{\rule[-1.0ex]{0pt}{3.1ex}}
\newcommand{\hfs}{\ensuremath{\frac{1}{2}}\fracStrut}
\newcommand{\scinot}[2]{\ensuremath{#1\times10^{#2}}}
\newcommand{\dee}{\mathrm{d}}
\newcommand{\dye}{\partial}
\newcommand{\diff}[2]{\frac{\dee #1}{\dee #2}}
\newcommand{\pdiff}[2]{\frac{\dye #1}{\dye #2}}
\newcommand{\Real}{\mathbb{R}}
\newcommand{\Complex}{\mathbb{C}}
\newcommand{\m}[1]{\mathbf{#1}}
\newcommand{\mA}{\m{A}}
\newcommand{\mI}{\m{I}}
\newcommand{\mK}{\m{K}}
\newcommand{\mL}{\m{L}}
\newcommand{\mX}{\m{X}}
\newcommand{\matalpha}{\boldsymbol{\upalpha}}
\newcommand{\matgamma}{\boldsymbol{\upgamma}}
\newcommand{\matbeta}{\boldsymbol{\upbeta}}
\newcommand{\matmu}{\boldsymbol{\upmu}}
\newcommand{\matlambda}{\boldsymbol{\uplambda}}
\renewcommand{\v}[1]{\boldsymbol{#1}}
\newcommand{\transpose}{^\mathrm{T}}
\newcommand{\bT}{\v{b}\transpose}
\newcommand{\vb}{\v{b}}
\newcommand{\vc}{\v{c}}
\newcommand{\vd}{\v{d}}
\newcommand{\ve}{\v{e}}
\newcommand{\vu}{\v{u}}
\newcommand{\vv}{\v{v}}
\newcommand{\vy}{\v{y}}
\newcommand{\matlab}{{\sc Matlab}\xspace}
\newcommand{\cvx}{{\texttt{CVX}}\xspace}
\newcommand{\sedumi}{{\texttt{SeDuMi}}\xspace}
\newcommand{\sdpt}{{\texttt{SDPT3}}\xspace}
\newcommand{\code}[1]{\textsf{#1}}
\newcommand{\sspcoef}{\mathcal{C}}

\newcommand{\tL}{\tilde{L}}
\newcommand{\tF}{\tilde{F}}
\newcommand{\tR}{\tilde{R}}
\newcommand{\tr}{\tilde{r}}
\newcommand{\tz}{\tilde{z}}
\newcommand{\tbeta}{\tilde{\beta}}

\newcommand{\Hopt}{H_\textup{opt}}
\newcommand{\hmax}{h_\textup{max}}
\newcommand{\hmin}{h_\textup{min}}
\newcommand{\Ropt}{R_\textup{opt}}

\newcommand{\real}{\operatorname{Re}}
\newcommand{\imag}{\operatorname{Im}}

%% file: tables/realaxis-123410.tex
\begin{table}\centering
\begin{tabular}{l|lllll} 
       & \multicolumn{5}{c}{$\Hopt/s^2$} \\ 
Stages & $p=1$ & $p=2$ & $p=3$ & $p=4$ & $p=10$ \\ \hline 
 1  & 2.000  &  &  &  & \\ 
 2  & 2.000  & 0.500  &  &  & \\ 
 3  & 2.000  & 0.696  & 0.279 &  & \\ 
 4  & 2.000  & 0.753  & 0.377  & 0.174  & \\ 
 5  & 2.000  & 0.778  & 0.421  & 0.242  & \\ 
 6  & 2.000  & 0.792  & 0.446  & 0.277  & \\ 
 7  & 2.000  & 0.800  & 0.460  & 0.298  & \\ 
 8  & 2.000  & 0.805  & 0.470  & 0.311  & \\ 
 9  & 2.000  & 0.809  & 0.476  & 0.321  & \\ 
 10  & 2.000  & 0.811  & 0.481  & 0.327  & 0.051 \\ 
 15  & 2.000  & 0.817  & 0.492  & 0.343  & 0.089 \\ 
 20  & 2.000  & 0.819  & 0.496  & 0.349  & 0.120 \\ 
 25  & 2.000  & 0.820  & 0.498  & 0.352  & 0.125 \\ 
 30  & 2.001  & 0.821  & 0.499  & 0.353  & 0.129 \\ 
 35  & 2.000  & 0.821  & 0.499  & 0.354  & 0.132 \\ 
 40  & 2.000  & 0.821  & 0.500  & 0.355  & 0.132 \\ 
\hline 
\end{tabular} 
\caption{Scaled size of real axis interval inclusion for optimized methods.\label{tbl:realaxis123410}}
\end{table}

%% file: tables/imagaxis-1234.tex
\begin{table}\centering
\begin{tabular}{l|llll} 
       & \multicolumn{4}{c}{$\Hopt/s$} \\ 
Stages & $p=1$ & $p=2$ & $p=3$ & $p=4$ \\ \hline 
 2  & 0.500  &  &  & \\ 
 3  & 0.667  & 0.667  & 0.577 & \\ 
 4  & 0.750  & 0.708  & 0.708  & 0.707 \\ 
 5  & 0.800  & 0.800  & 0.783  & 0.693 \\ 
 6  & 0.833  & 0.817  & 0.815  & 0.816 \\ 
 7  & 0.857  & 0.857  & 0.849  & 0.813 \\ 
 8  & 0.875  & 0.866  & 0.866  & 0.866 \\ 
 9  & 0.889  & 0.889  & 0.884  & 0.864 \\ 
 10  & 0.900  & 0.895  & 0.895  & 0.894 \\ 
 15  & 0.933  & 0.933  & 0.932  & 0.925 \\ 
 20  & 0.950  & 0.949  & 0.949  & 0.949 \\ 
 25  & 0.960  & 0.960  & 0.959  & 0.957 \\ 
 30  & 0.967  & 0.966  & 0.966  & 0.966 \\ 
 35  & 0.971  & 0.971  & 0.971  & 0.970 \\ 
 40  & 0.975  & 0.975  & 0.975  & 0.975 \\ 
 45  & 0.978  & 0.978  & 0.978  & 0.977 \\ 
 50  & 0.980  & 0.980  & 0.980  & 0.980 \\ 
\hline 
\end{tabular} 
\caption{Scaled size of imaginary axis inclusion for optimized methods.\label{tbl:imagaxis1234}}
\end{table}